\documentclass[a4paper,12pt,centertags]{amsart}
\usepackage{amssymb}
\usepackage{mathtools}
\usepackage{xspace}
\usepackage[a4paper,centering]{geometry}
\usepackage[pdfdisplaydoctitle,colorlinks,urlcolor=blue,linkcolor=blue,citecolor=blue]{hyperref}
\usepackage{tikz}
\newtheorem{theorem}{Theorem}[section]
\newtheorem{lemma}[theorem]{Lemma}
\newtheorem{proposition}[theorem]{Proposition}
\newtheorem{corollary}[theorem]{Corollary}
\theoremstyle{definition}

\theoremstyle{remark}
\newtheorem{remark}[theorem]{Remark}
\newtheorem{example}[theorem]{Example}
\numberwithin{equation}{section}
\DeclareMathOperator{\Div}{div}
\newcommand{\e}{\operatorname{e}}
\newcommand{\im}{\mathrm{i}}
\newcommand{\N}{\mathbf{N}}
\newcommand{\Z}{\mathbf{Z}}
\newcommand{\R}{\mathbf{R}}
\newcommand{\C}{\mathbf{C}}
\newcommand{\E}{\mathbb{E}}
\newcommand{\Prob}{\mathbb{P}}
\newcommand{\eqdef}{\vcentcolon=}
\newcommand{\scal}[1]{\langle #1 \rangle}
\newcommand{\torus}{\mathbb{T}_2}
\newcommand{\B}{\mathcal{B}}
\newcommand{\Bt}{\smash{\widetilde\B}}
\newcommand{\T}{\mathcal{T}}
\newcommand{\X}{\mathcal{X}}
\newcommand{\J}{\mathcal{J}}
\newcommand{\Wz}{\smash{\mathring W}_\wp}
\newcommand{\Hz}{\smash{\mathring H}_\wp}
\newcommand{\Lz}{\smash{\mathring L}}
\newcommand{\memo}[1]{
  \ensuremath{
    \framebox{\tiny\textbf{\kern-2pt\textsf{#1}}\kern-2pt}
  }
  \xspace
}
\hypersetup{%
  pdftitle={2D surface growth with noise},
  pdfauthor={D. Blomker, M. Romito},
  pdfcreator={D. Blomker, M. Romito}
}
\begin{document}
  \title[2D surface growth with noise]{Local existence and uniqueness
    for a two--dimensional surface growth equation with space--time white noise}
  \author[D. Bl\"omker]{Dirk Bl\"omker}
    \address{Institut f\"ur Mathematik\\ Universit\"at Augsburg\\ D-86135 Augsburg, Germany}
    \email{\href{mailto:dirk.bloemker@math.uni-augsburg.de}{dirk.bloemker@math.uni-augsburg.de}}
    \urladdr{\url{http://www.math.uni-augsburg.de/ana/bloemker.html}}
  \author[M. Romito]{Marco Romito}
    \address{Dipartimento di Matematica, Universit\`a di Pisa, Largo Bruno Pontecorvo 5, I--56127 Pisa, Italia}
    \email{\href{mailto:romito@dm.unipi.it}{romito@dm.unipi.it}}
    \urladdr{\url{http://www.dm.unipi.it/pages/romito}}
  \thanks{Part of the work was done at the Newton institute for Mathematical
    Sciences in Cambridge (UK), whose support is gratefully acknowledged,
    during the program \emph{Stochastic partial differential equations}.}
  \date{July 14, 2013}
  \begin{abstract}
    We study local existence and uniqueness for a surface growth model with
    space--time white noise in 2D. Unfortunately, the direct fixed-point argument
    for mild solutions fails here, as we do not have sufficient
    regularity for the stochastic forcing.  Nevertheless, one can give a 
    rigorous meaning to the stochastic PDE and show uniqueness of solutions in that setting.
    Using spectral Galerkin method and any other types of regularization of the noise,
    we obtain always the same solution.
  \end{abstract}
\maketitle
\section{Introduction}

We study local existence and uniqueness of the following equation 
\begin{equation}\label{e:2dsg}
  \partial_t h + \Delta^2 h + \Delta|\nabla h|^2
    = \eta
\end{equation}
subject to periodic boundary conditions on  $[0,2\pi]^2$ and with space--time
white noise $\eta$. This equation  arises in the theory of amorphous surface
growth, see for example \cite{Rai-a,Rai-b}, and it has been considered also
in the theory of ion sputtering. The equation is simplified in the sense that
we have left out lower order terms that can easily be handled and do not present
any obstacle in the theory of local existence.

A thorough analysis of the one--dimensional version of the problem has been
given in \cite{BloRom2009}, where the general theory introduced in
\cite{FlaRom2008} has allowed to prove the existence of Markov solutions to
the equation. Moreover each of these solutions converges to a unique equilibrium
distribution.

In order to complete the same program for the (physically relevant)
two--dimensional case, there are several problems that need to be faced.
\begin{itemize}
  \item When dealing with space--time white noise it turns out that
    the expected smoothness of the solution is not enough to define
    the non--linear term.
  \item In contrast with the one--dimensional case, existence of global
    weak solutions is harder, due to the lack of reasonable energy estimates.
    Existence of weak solutions without noise has been proved in
    \cite{Wi2011}, using a--priori bounds derived from the estimate
    of $\int e^h\,dx$, which cannot be used for any approximation by  Galerkin methods.
  \item The weak--strong uniqueness principle ({i.e.} uniqueness of
    local solutions in the class of weak solutions) fails, and this
    is a serious obstacle to the application of the same ideas
    used in \cite{BloFlaRom2009}.
\end{itemize}

In this paper we address the problem of proving the existence
of local solutions for the equation forced by space--time white noise,
as well as the issue of low regularity for the non--linear term.
A preliminary existence result of local in time solution has been
already given in \cite{BloRom2012}, based on the ideas introduced
in \cite{KocTat2001}. Nevertheless, these methods where not able to 
treat the physically relevant case of space-time white noise. 

Here we follow an approach similar to
the one used in \cite{DapDeb2002} for a similar singular
two--dimensional problem with space--time white noise.
One key difference is that we cannot rely on an explicitly 
given invariant measure.
The idea is to decompose the solution in a rough part having
the low regularity dictated by the forcing and a remainder,
slightly smoother. The non--linearity for the rough term is
then defined as the limit of  cut--offs via spectral Galerkin methods, 
thanks to the underlying Gaussian nature of the processes involved. This,
roughly speaking, corresponds to a re--normalization of the
non--linearity, but without any additional term, due to the
fact that the Laplace in front of the gradient squared
kills the infinite constant. The method works even for rougher
noise at the price, of a lower level of regularity. For even rougher noise 
the remainder fails to be regular enough, and then we 
need an additional term in the expansion of the solution.

To be more precise, we interpret the solution as $h = z + v$,
where $z$ solves the corresponding linear stochastic equation
(where the non--linearity is dropped) and $v$ is the remainder.
The term $z$ is a function valued process. Due to regularization of the bi--Laplacian
it is even continuous in space and time, but it fails to have a derivative. 
The remainder $v$ is then given by the mild solution
\[
  v(t)
    =  \e^{-tA}h_0
     - \int_0^t\e^{-(t-s)\Delta^2} \Delta(
       |\nabla v|^2 + 2\nabla v\cdot\nabla z + |\nabla z|^2)\,dt
     + z(t).
\]
Since $v$ is regular enough, the double product $\nabla v\cdot\nabla z$
is well defined. The problems originate from the ``squared distribution''
$|\nabla z|^2$. Once this term is properly defined as a limit of
spectral approximations, for instance, we can work out a fixed point argument
similar to \cite{FujKat1964} (see also \cite{BloRom2009}).

Recently there is an interest in the analysis of
non--linear PDEs that, like the one presented here,
are forced by noise rough enough so that the non--linear
term in principle is not well--defined. The meaning of
the non--linearity is then recovered through probability.
We refer to \cite{Hai2012,HaiWeb2013}. Two recent
papers \cite{GubImkPer2012,Hai2013} have proposed
general and powerful methods that apply to our equation
\eqref{e:2dsg} as well as to more difficult problems,
where for instance the re--normalized infinite constant
shows up in the equation. The method we have used here,
based on Fourier expansion, works very well for our
problem \eqref{e:2dsg} and we believe it is, at least
for this problem, neat and simple.
\subsection{Notations}

Let $\torus$ be the two dimensional torus, understood as
$\torus = [-\pi,\pi]^2$ with identification of the borders.
Consider the complex Fourier basis of $L^2(\torus;\C)$,
defined by $e_k = (2\pi)^{-1}\e^{\im k\cdot x}$, for $z\in\Z^2$.
Notice that if $u\in L^2(\torus)$ is real valued and
$u = \sum u_k e_k$, then  $u_k = \bar u_{-k}$.
Define $\Z^2_\star = \Z^2\setminus\{0\}$,
$\Z^2_+ = \{k\in\Z^2: k_1>0\text{ or }k_1=0,k_2>0\}$
and $\Z^2_- = -\Z^2_+$.

For $s\in\R$ and $p\geq1$, denote by $W^{s,p}_\wp$ the
space of $2\pi$--periodic $W^{s,p}$ functions, and by
$\Wz^{s,p}$ its sub--space of functions with zero mean.
We will use the notation $H^s_\wp$, $\Hz^s$ and
$\Lz^2$ when $p=2$.
Such spaces can be defined either as the closure in the
corresponding norm of smooth periodic functions, or as spaces
on $\R^2$ with weight, or by interpolation. We refer to
\cite[Chapter 3]{SchTri1987} for details on the definition
as well as for their properties. In particular, we will
use that the dual $(W^{s,p}_\wp)'$ is $W^{-s,q}_\wp$,
where $q$ is the H\"older conjugate exponent of $p$.
Moreover the standard Sobolev embeddings hold, namely
$W^{s,q}_\wp\subset W^{r,p}_\wp$ if $r\leq s$
and $r-\tfrac{d}{p}\leq s-\tfrac{d}{q}$.

Denote by $(S(t))_{t\geq0}$ the semigroup generated in $\Lz^2$ 
by $-\Delta^2$ with domain of definition $\Hz^4$. 
Set for real valued functions
$u_1,u_2$,
\[
  \B(u_1,u_2)
    = \Delta(\nabla u_1\cdot\nabla u_2)\;.
\]
It will be useful, for our purposes, to extend the definition
of $\B$ over complex valued functions as
$\B(u_1,u_2) = \Delta(\nabla u_1\cdot\overline{\nabla u_2})$.
With this position $\B$ coincides with the previous definition
for real valued functions and it is Hermitian, namely
$\B(u_2,u_1) = \overline{\B(u_1,u_2)}$. 

In the rest of the paper we shall adopt the sloppy habit to
use the same symbol for numbers that depend only on universal
constants and that can change from line to line of an inequality.
\section{Main results}

\subsection{Existence of mild solutions}\label{s:mild1}

Our first main results shows existence of a local solution
of \eqref{e:2dsg}. The solution is interpreted in the mild
formulation
\begin{equation}\label{e:hmild}
  h(t)
    = S(t)h_0
        - \int_0^t S(t-s)\B(h,h)\,ds
        - \int_0^t S(t-s)\,dW_s,
\end{equation}
for a suitable initial condition $h_0$.
Denote by $z(t) = \int_0^t S(t-s)\,dW_s$ the stochastic convolution.
The above mild formulation can be recast in terms of $v = h - z$ as
\[
  v(t)
    = S(t)h_0
        - \int_0^t S(t-s)\B(v,v)\,ds
        - \int_0^t S(t-s)\bigl(\B(z,z) + 2\B(z,v)\bigr)\,ds.
\]
Since we do not expect the term $\B(z,z)$ to be well defined,
given the regularity of the stochastic convolution (see Lemma
\ref{l:fail}, we replace it by a suitable extension $\Bt(z,z)$
defined in Section \ref{s:z} as the limit of spectral Galerkin
approximations of $z$.

In conclusion a mild solution $h$ of \eqref{e:2dsg} is
a random process such that $v = h - z$ is a mild solution
in the sense that
\begin{equation}\label{e:vmild}
  v(t)
    = S(t)h_0
        - \int_0^t S(t-s)\B(v,v)\,ds
        - \int_0^t S(t-s)\bigl(\Bt(z,z) + 2\B(z,v)\bigr)\,ds.
\end{equation}
Given $\rho>0$, $\epsilon>0$ and $T>0$, define
\[
  \|u\|_{\epsilon,T}
    \eqdef\sup_{t\leq T} t^{\frac\epsilon4}\|u(t)\|_{H^{1+\epsilon}}.
\]
We will find a solution of \eqref{e:vmild} by means of a fixed point
argument in the space
\[
  \X(\epsilon,\rho,T)
    \eqdef\{v\in C([0,T];\Lz^2(\torus)):
      \|v\|_{\epsilon,T}\leq\rho\}
\]
for suitable $\rho,T$.
\begin{theorem}\label{t:mild}
  Let $h_0\in \Hz^1$ and $\epsilon\in(0,\tfrac12)$.
  Given a cylindrical Wiener process $W$ on $\Lz^2(\torus)$,
  there exist a stopping time $\tau_{h_0}$ and a solution
  $h$ of the mild formulation \eqref{e:hmild} defined on $[0,\tau_{h_0})$,
  such that $h\in C([0,\tau_{h_0});\Lz^2(\torus))$ and
  $h - z\in C((0,\tau_{h_0});\Hz^{1+\epsilon})$.
  Moreover, $\Prob[\tau_{h_0}>0]=1$.
\end{theorem}
The proof is given later in Section \ref{t:mild}.
Note that solutions are unique up to the minimum of both 
their stopping times in the space $\X(\epsilon,\rho,T)$.
Moreover, by standard methods one can continue uniquely the solution 
as continuous $\Hz^{1+\epsilon}$-valued solutions, until they blow up.
\subsection{Other regularizations}

Here we consider regularizing methods 
different from the spectral Galerkin method used to define $\Bt(z,z)$.
We give an abstract criterion and examples of its application.

We first define what we mean by regularization of $z$. Let $\Phi$
be a bounded operator on $\Lz^2(\torus)$ and define
\[
  z^\Phi(t)
    = \int_0^t S(t-s)\Phi\,dW(s)
    = \sum_{k\in\Z^2_\star}\int_0^t \e^{-(t-s)|k|^4}\,d\beta^\Phi_k(s)e_k,
\] 
where the $\beta_k^\Phi(t)=\langle \Phi W(t),e_k\rangle$ are
(not necessarily independent) Brownian motions with variance $\|\Phi e_k\|^2_{L^2}$. We suppose that
the regularized process $z^\Phi$ defined above is smooth enough in order
to define $B(z^\Phi,z^\Phi)$ uniquely. A sufficient condition that
ensures this statement is given in the following lemma.
\begin{lemma}\label{l:admissible}
  Let $\Phi$ and $z^\Phi$ as above. Assume that
  for every $k\in\Z^2_\star$,
  \begin{equation}\label{e:absconv}
    \sum_{m+n=k} \frac{\|\Phi e_m\|_{L^2}\|\Phi e_n\|_{L^2}}{|m|\,|n|}
      <\infty,
  \end{equation}
  then $B(z^\Phi,z^\Phi)$ is well--defined as an element
  of $\Hz^{-2-\gamma}$ for every $\gamma>0$.  
  In particular \eqref{e:absconv} holds if
  $\sum_m |m|^{-2}\|\Phi e_m\|^2_{L^2}<\infty$.
\end{lemma}
Having approximations in mind, we turn to sequences $(\Phi_N)_{N\in\N}$
of bounded operators satisfying \eqref{e:absconv} and we analyse
under which conditions they provide a ''good'' approximation of
the process $z$. By ''good'' we mean that the quantities involved
in the definition of $\Bt(z,z)$ and in the proof of Theorem
\ref{t:mild} should be well approximated by the corresponding
quantities for the sequence $(z^{\Phi_N})_{N\geq1}$.
The first result gives sufficient conditions that ensure
convergence in $L^p_\text{loc}([0,\infty);W^{s,p})$.
The technical assumptions on $\Phi_ N$ basically states, that they converge 
in a weak sense to the identity, and that the off-diagonal terms of the operators are not too large.
\begin{theorem}\label{thm:or2}
  Let $(\Phi_N)_{N\geq1}$ be a sequence of bounded operators
  on $\Lz^2(\torus)$ such that
  \begin{itemize}
    \item for every $m,n\in\Z_\star^2$,
      \[
        \langle\Phi_N^\star\Phi_N e_m,e_n\rangle
          \longrightarrow\delta_{m,n},
     \]
      where the Kronecker-Delta is given by $\delta_{m,n}=1$ if $m=n$ and zero otherwise,
    \item there is $\gamma\in(0,1)$ such that
      \begin{equation} \label{e:thmor2}
       \sum_{m,n}\sup_{N\in\N} \Big\{\frac{|\langle(\Phi_N-I) e_m,(\Phi_N-I)e_n\rangle|}
            {(|m|+|n|)^{4-2\gamma}}\Big\}
          <\infty.
     \end{equation}
  \end{itemize}
  Then for every $s\in(0,\gamma)$, $p\geq1$ and $T>0$,
  \[
    \E[\|z^{\Phi_N} - z\|_{L^p([0,T];W^{s,p})}^p]
      \longrightarrow0,
        \qquad N\to\infty.
  \]
\end{theorem}
Our second result gives some sufficient conditions that
ensure that different approximations give the same
limit non--linearity. The particular choice of the
Galerkin truncations operators yields conditions
for the limit of a generic sequence $(\Phi_N)_{N\in\N}$
to the limit non--linearity defined in Section \ref{s:z}.
\begin{theorem}\label{thm:konsist}
  Let $(\Phi_N)_{N\in\N}$ and $(\Psi_N)_{N\in\N}$
  be two sequences of regularizing operators such that for $N\to \infty$
  \begin{equation}\label{e:Phiconv}
    \begin{gathered}
    \langle\Phi_N e_m,\Phi_Ne_n\rangle\to\delta_{m,n},\qquad
    \langle\Psi_N e_m,\Psi_Ne_n\rangle\to\delta_{m,n},\\
    \langle\Psi_N e_m,\Phi_Ne_n\rangle\to\delta_{m,n},
    \end{gathered}
  \end{equation}
  for every $m,n\in\Z^2_\star$. Let 
  \[
    c_{mn}= \sup_{N\in\N} \{
       |\langle\Phi_N e_m,\Phi_Ne_n\rangle|
        + |\langle\Psi_N e_m,\Psi_Ne_n\rangle|
        + |\langle\Phi_N e_m,\Psi_Ne_n\rangle| \}
  \]
  for $m,n\in\Z^2_\star$, and assume that for some $\gamma>0$,
  \begin{equation}\label{e:Phibound1}
    \sum_{k\in\Z^2_\star}|k|^{-2\gamma}\sum_{m+n=k}
        \frac{c_{mn}}{|n|^3|m|^3}
      <\infty,
  \end{equation}
  and
  \begin{equation}\label{e:Phibound2}
    \sum_{k\in\Z^2_\star}|k|^{-2\gamma}
        \sum_{\substack{m_1+n_1=k\\ m_2+n_2=k}}
        \frac{c_{m_1m_2} c_{n_1n_2}}{|m_1|^3|m_2|^3|n_1|^3|n_2|^3}
      <\infty.
  \end{equation}
  Then for all $T>0$ and $q\geq1$,
  \[
    \E\bigl[\|\B(z^{\Phi_N} ,z^{\Phi_N})
        - \B(z^{\Psi_N} ,z^{\Psi_N})\|^q_{L^q([0,T], H^{-2-\gamma})}\bigr]
      \longrightarrow 0.
  \]
\end{theorem}
With the above results at hand, we verify that a convergence
of a regularization $(\Phi_n)_{n\in\N}$ that leads to the
convergence of $z^{\Phi_n}$ to $z$  and of $\B(z^{\Phi_n},z^{\Phi_n})$
to $\tilde\B(z,z)$ result in the solution $v^{\Phi_n}=h^{\Phi_n}-z^{\Phi_n}$
of the regularized problem converging to $v=h-z$, the solution given by
Theorem \ref{t:mild}, in probability. We give only one possible version
of the result. Other versions may be obtained by working on different
function spaces.

Given an initial condition $h_0\in\Hz^1(\torus)$, let $v$ be the
process given by the mild formulation \eqref{e:vmild}. Define
for every $R>0$ the stopping time 
\[
  \tau^R
    =\inf\{t>0:\ \|v(t)-S(t)h_0\|_{H^{1+\epsilon}}\geq R\},
\]
and $\tau^R = \infty$ if the above set is empty. By its definition,
it is immediate to see that $\tau^R\leq\tau_{h_0}$, where
$\tau_{h_0}$ is the life--span of $v$.
\begin{theorem}\label{thm:stability}
  Let $h_0\in\Hz^1(\torus)$.
  Let $\Phi_N$ be a sequence of regularizing operators such 
  that $z^{\Phi_N}\in C^0([0,\infty),\Hz^{1+\epsilon})$ for all $N\in\N$ and  
  fix
  \[\epsilon \in (0,\frac12), \quad 
   \alpha=1-\frac{\epsilon}2, \quad
   q>\frac4\epsilon, \quad
  \beta\in(2,3-\epsilon),\ \text{and}\ q'>\frac4{3-\beta-\epsilon}\;.
  \]
  If
  \begin{equation}\label{e:convstab}
    \E\|z^{\Phi_N} - z \|_{L^q([0,1],W^{\alpha,q})}
        + \E\|\B(z^{\Phi_N},z^{\Phi_N})-\tilde\B(z,z)\|_{L^{q'}([0,1],H^{-\beta})}
      \to 0,
  \end{equation}
  as $N\to\infty$, then
  \[
    \sup_{[0,1\wedge\tau_R]} \|v-v^{\Phi_N}\|_{H^{1+\epsilon}}
        \longrightarrow 0,
  \] 
  in probability.
\end{theorem}
The proof of these results is given in Section \ref{sec:stab}.
Here we illustrate examples of applications of the results presented above.
\begin{remark}
  For every $N\geq1$ define $\Phi_Ne_m = e_m$ if $|m|\leq N$,
  and $0$ otherwise. The spectral truncations $\Phi_N$ are clearly
  regularizing and all assumptions of Theorems \ref{thm:or2},
  and \ref{thm:konsist} hold true. Indeed, the two theorems
  find a non--trivial application when one needs to control
  the off--diagonal terms.
\end{remark}
\begin{example}\label{exp:or1}
  Given a non--negative smooth function $q$ with support contained
  in a small neighbourhood of the origin and such that $\int q(x)\,dx=1$,
  let $q_N$ be the periodic extension on $\torus$ of $z\mapsto q(Nz)$.
  Let 
  \[
    \Phi_N f(x)
      = \int_{[0,2\pi]^2} N^2 q_N(x-y)f(y)\,dy
  \]
  The operators $\Phi_N$ are self--adjoint and diagonal in the
  Fourier basis. Denote by $\phi_k^N$ the eigenvalues
  of $\Phi_N$. These are (up to constant) given by the Fourier
  coefficients of $z\mapsto N^2 q(Nz)$.
  
  Assume that $q\in H^\eta$ for some $\eta>0$ and that
  $\phi_k^N\to1$ as $N\to\infty$, which is easy to verify. 
  Then it is straightforward to check
  that the assumptions of Lemma \ref{l:admissible} and of
  Theorems \ref{thm:or2} and \ref{thm:konsist} are verified.
  We comment in more details in Example \ref{exp:or1bis}.
\end{example}
\begin{example}\label{exp:or2}
  Here we study a non-diagonal case, which is for instance given by noise not homogeneous in space \cite{DB05}. 
  Define 
  \[
    \Phi_N f(x)
      = \int_{[0,2\pi]^2} {N^2} q_N\big( x, y)\big)f(y)\,dy
  \]
  with a kernel $q_N$ which determined by a
  non-negative, smooth  $q$ such that
  the support of $q$  is contained in a small neighbourhood 
  of the diagonal $x=y$ such that $\int_{\torus\times\torus}q(x,y)dxdy=1$.
  The kernel $q_N$ 
  is the periodic extension on $\torus\times\torus$ of $(\xi,\eta)\mapsto q(N\xi,N\eta)$.

    We can now again check all assumptions of Lemma \ref{l:admissible} 
    and  Theorem \ref{thm:or2} and \ref{thm:konsist}.
  As before, let $\Psi^{(N)}=\pi_N$ be the projection onto the first
  Fourier modes. Again, \eqref{e:absconv} is true, once $q$ is
  sufficiently smooth.
   
  The bounds in (\ref{e:Phibound1}) and  (\ref{e:Phibound2})  are easy to establish, 
  as we verify later that $c_{m.n}$ is uniformly bounded.
  The crucial condition is (\ref{e:thmor2}), which requires 
  some work and does not seem hold for arbitrary kernels. 
  We comment on all these assumptions in detail later
  in Example \ref{exp:or2bis}.
\end{example}
\subsection{Rougher noise}

As it is apparent by the previous sections,
space--time white noise is the borderline case between
the standard theory for mild solutions and the additional
work summarized by Theorem \ref{t:mild}. It is then possible
to consider rougher noise.

In view of the computations needed to define $\Bt(z,z)$
(Lemma \ref{l:wick}) it is reasonable to consider
a simplified case, namely when the covariance operator we
apply to white noise is diagonal in the Fourier basis.
Consider a bounded linear operator $\Phi$ on $\Lz^2(\torus)$
and assume for the rest of this section the following properties,
\begin{itemize}
  \item $\Phi\e_k = \phi_k e_k$ for every $k\in\Z^2_\star$,
  \item there is $\beta>0$ such that
    $|\phi_k|^2\leq c|k|^\beta$
     for every $k\in\Z^2_\star$.
\end{itemize}
This situation is similar to Example \ref{exp:or1} before,
when we consider kernels $q$ given by a distribution
instead of a function.

The value $\beta=0$ is morally the space--time white noise.
Moreover, the definition of $\Bt(z,z)$ imposes a structural
restriction that limits the range of possible values of
$\beta$ to $\beta<1$ (see Remark \ref{r:threshold}).

The same ideas of Section \ref{s:mild1}, when slightly modified
to take into account the parameter $\beta$, lead to the
following result.
\begin{theorem}\label{t:mild2}
  Assume $\beta\in(0,\tfrac23)$ and let $h_0\in \Hz^1$
  and $\epsilon\in(\tfrac\beta2,(1-\beta)\wedge\tfrac12)$.
  Given a cylindrical Wiener process $W$ on $\Lz^2(\torus)$,
  there exist a stopping time $\tau_{h_0}$ and a solution
  $h$ of \eqref{e:hphi} understood as $h = v + z$, where
  $z$ is given by \eqref{e:zphi} and $h-z$ satisfies the
  mild formulation \eqref{e:vmild} on $[0,\tau_{h_0})$.
  Moreover, $h\in C([0,\tau_{h_0});\Lz^2(\torus))$,
  $h - z\in C((0,\tau_{h_0});\Hz^{1+\epsilon})$
  and $\Prob[\tau_{h_0}>0]=1$.
\end{theorem}
The restriction $\beta<\tfrac23$ is due to the term
$\B(v,z)$ in the mild formulation \eqref{e:vmild}.
When the noise is too rough, the auxiliary function
$v$ is not enough regular to ensure that the product
$\B(v,z)$ is well--defined. 

Assume now $\beta\in[\tfrac23,1)$. To overcome the
difficulty caused by the poor regularity of both $v$
and $z$, we add a term in the second Wiener chaos
in the decomposition of $h$, namely $h = u + \zeta + z$,
where $\zeta$ solves
\[
  \dot\zeta + A\zeta + \Bt(z,z) = 0,
    \qquad
  \zeta(0) = 0,
\]
and $u$ is the mild solution of
\[
  \dot u + Au + \B(u,u) + 2\B(u,z) + 2\B(u,\zeta)
     + 2\Bt(\zeta,z) + \B(\zeta,\zeta) = 0,
\]
with initial condition $u(0) = h(0)$. To this end
we need to suitably define $\Bt(\zeta,z)$ as we
have already done for $\Bt(z,z)$, by exploiting the
cancellations in the expectations of these processes.
This gives no gain for $\zeta$ (we have already
``used'' the cancellation to define $\Bt(z,z)$)
but it is effective both in improving the regularity
of $\B(\zeta,\zeta)$ (with respect to what we
would get from standard multiplication theorems
in Sobolev spaces), and in defining $\Bt(\zeta,z)$.

Actually, the approach through the higher Wiener chaos
expansion of $h$ can be used for any value of
$\beta\in(0,1)$. Indeed, it is sufficient to define
$\Bt(\zeta,z) = \B(\zeta,z)$ whenever the latter
term is well defined (see Remark \ref{r:extend}).
We are able then to prove the following result.
\begin{theorem}\label{t:mild3}
  Assume $\beta\in(0,1)$ and let $h_0\in \Hz^1$
  and $\epsilon\in(\tfrac\beta2,\tfrac12)$.
  Given a cylindrical Wiener process $W$ on $\Lz^2(\torus)$,
  there exist a stopping time $\tau_{h_0}$ and a solution
  $h$ of \eqref{e:hphi} understood as $h = u + \zeta + z$,
  where $z$ is given by \eqref{e:zphi}, $\zeta$ by
  \eqref{e:zeta} and $u= h-z-\zeta$ satisfies the
  mild formulation \eqref{e:higher} on $[0,\tau_{h_0})$.
  Moreover, $h\in C([0,\tau_{h_0});\Lz^2(\torus))$,
  $h - z - \zeta\in C((0,\tau_{h_0});\Hz^{1+\epsilon})$
  and $\Prob[\tau_{h_0}>0]=1$.
\end{theorem}
\section{The stochastic convolution}\label{s:z}

Let $z$ be the stochastic convolution
\[
  z(t)
    = \int_0^t S(t-s)\,dW_s,
\]
namely the solution of
\[
  dz + Az\,dt = dW,
\]
with initial condition $z(0) = 0$ and zero mean. The stochastic convolution
can be expanded in the complex Fourier basis,
\begin{equation}\label{e:zkappa}
  z = \sum_{k\in\Z^2_\star}z_k e_k,
    \qquad
  z_k(t) = \int_0^t\e^{-|k|^4(t-s)}\,d\beta_k(s),
\end{equation}
where $\beta_k = \scal{W_t,e_k}$, $\beta_{-k} = \bar\beta_k$,
and $(\beta_k)_{k\in\Z^2_+}$
is a sequence of independent complex--valued standard Brownian motions.

Due to the bi--Laplace operator, the stochastic convolution is
function--valued. On the other hand the stochastic convolution
is not sufficiently regular to define the non--linear term
$\Delta|\nabla z|^2$ as a function (and neither as a
distribution), see Lemma \ref{l:fail} below. It turns out that,
suitably defined, the term $\Delta|\nabla z|^2$ makes
sense.
\begin{lemma}\label{l:fail}
  For every $t>0$,
  \[
    \E[\|z(t)\|_{H^1}^2]
      = \infty
  \]
  and $z\not\in\Hz^1$ for all times, almost surely.
\end{lemma}
\begin{proof}
  Using the explicit representation of $z$ in Fourier series,
  \[
    \E[\|\nabla z(t)\|_{L^2}^2]
      = \sum_{k\in\Z^2_\star}|k|^2\int_0^t\e^{-2|k|^4(t-s)}\,ds
      = \sum_{k\in\Z^2_\star}\frac1{2|k|^2}(1 - \e^{-2|k|^4t})
      = \infty.
  \]
  The almost sure statement follows from Gaussianity
  \cite[Theorem 2.5.5]{Bog1998}.
\end{proof}
\subsection{Regularity in Sobolev spaces}

Lemma \ref{l:fail} above shows that the gradient of $z$ is not
defined. On the other hand, $z$ has a ``fractional'' derivative
of any order smaller than one.
\begin{proposition}\label{p:zreg}
  For every $p\geq1$ and $s\in(0,1)$, 
  \[
    \sup_{t>0} \E\bigl[\|z(t)\|^p_{W^{s,p}}\bigr]
      <\infty.
  \]
\end{proposition}
\begin{proof}
  Use the Fourier representation of $z$ to get,
  \[
    \begin{multlined}[.9\linewidth]
      \E[|z(t,x)-z(t,y)|^2]
        \leq c\sum_{k\in\Z^2_\star}\E[|z_k(t)|^2]|e_k(x)-e_k(y)|^2\leq {}\\
        \leq c\sum_{k\in\Z^2_\star}\frac{1\wedge|k\cdot(x-y)|^2}{|k|^4}
        \leq c|x-y|^2\log(8\pi|x-y|^{-1}),
    \end{multlined}
  \]
  where to estimate the last sum on the right hand side of the formula above
  one can split in the two parts $|k|\geq |x-y|^{-1}$ and
  $|k|\leq |x-y|^{-1}$. By Gaussianity, for every $p\geq1$,
  $\E[|z(t,x)-z(t,y)|^p]\leq c_p|x-y|^p\log^{p/2}(8\pi|x-y|^{-1})$.
  Therefore,
  \[
    \E\Bigl[\iint \frac{|z(t,x)-z(t,y)|^p}{|x-y|^{2+sp}}\,dx\,dy\Bigr]
      \leq c\iint\frac{\log^{p/2}(8\pi|x-y|^{-1})}
        {|x-y|^{2-(1-s)p}}\,dx\,dy
      <\infty.\qedhere
  \]
\end{proof}
\begin{remark}
  The regularity in time stated in the previous proposition can
  be improved, with standard arguments, to $L^\infty$ or even
  H\"older, but we will not use this fact in the paper. 
\end{remark}
\subsection{The non--linearity for the stochastic convolution}

If $u = \sum_{k\in\Z^2}u_k e_k$ and $v = \sum_{k\in\Z^2} v_k e_k$
are real valued, the non--linear term can be formally written in
terms of the Fourier coefficients as
\[
  \B(u,v)
    = \sum_{k\in\Z^2_\star}
        |k|^2 \Bigl(\sum_{m+n=k} m\cdot n\ u_m v_n\Bigr)e_k.
\]
Consider the stochastic convolution $z$ and set for every $k\in\Z^2_\star$,
\begin{equation}\label{e:Jk}
  J_k(t)
    = \sum_{m+n=k} m\cdot n\ z_m(t) z_n(t).
\end{equation}
Formally, $\B(z,z) = \sum_k |k|^2 J_k e_k$. Lemma \ref{l:fail}
immediately tells us that $J_0(t) = - \|\nabla z(t)\|_{L^2} = \infty$
almost surely. Likewise, an investigation of absolute convergence of $J_k(t)$
for $k\neq0$ yields
\[
  \E\Bigl[\smashoperator[r]{\sum_{m+n=k}}|m\cdot n|\,|z_m(t)\,z_n(t)|\Bigr]
    \geq \smashoperator{\sum_{\substack{m+n=k\cr m\neq n}}}|m\cdot n|\E[|z_m(t)|]\E[|z_n(t)|]
    \geq \smashoperator[r]{\sum_{\substack{m+n=k\cr m\neq n}}}\frac{c_t|m\cdot n|}{|m|^2|n|^2}
    = \infty.
\]

Following \cite{DapDeb2002}, we extend the definition of
the non--linearity $\B$ so that the terms $J_k(t)$, for
$k\neq0$, are convergent. This is possible due to
cancellations, since the $z_m$ are centred Gaussians.
The term $J_0(t) = -\|\nabla z(t)\|_{L^2}$ should
be the most problematic, since there is no hope to exploit
any cancellation. On the other hand it is constant in the
space variable and it is cancelled by the Laplace operator.

Given $N\geq 1$, let $H_N$ be the linear sub--space
of $L^2(\torus)$ spanned by $(e_k)_{0<|k|\leq N}$.
Let $\pi_N$ be the projection of $L^2(\torus)$
onto $H_N$ and define
\[
  \B_N(u,v)
    = \B(\pi_N u, \pi_N v).
\]
We extend the operator $\B$ on the non--differentiable
function $z$ as the limit of the sequence $(\B_N(z,z))_{N\geq1}$
in suitable function spaces.
\begin{lemma}\label{l:wick}
  Let $z$ be the stochastic convolution. Then $(\B_N(z,z))_{N\geq1}$
  is a Cauchy sequence in $L^2(\Omega;\Hz^{-2-\gamma})$ for every
  $\gamma>0$. In particular, the limit $\Bt(z,z)$ is well--defined
  as an element of $\Hz^{-2-\gamma}$.
\end{lemma}
\begin{proof}
  Let $J_k^N$ be the term analogous to $J_k$ for $\pi_N z$.
  If $N\leq N'$,
  \[
    \E[|J_k^N(t) - J_k^{N'}(t)|^2]
      =\sum_{m_1+n_1=k}^{N\leftrightarrow N'}
         \sum_{m_2+n_2=k}^{N\leftrightarrow N'}
         m_1\cdot n_1\ m_2\cdot n_2
         \E[z_{m_1}z_{n_1}\bar z_{m_2}\bar z_{n_2}],
  \]
  where by the symbol $N\leftrightarrow N'$ in the sum over $m,n$
  we mean that the sum is extended only over those indices $m,n$
  that satisfy $N<|m|\vee|n|\leq N'$.
  
  The sequence $(z_m)_{m\in\Z^+}$ is a family of independent
  centred  Gaussian random variables. Moreover
  $\bar z_m = z_{-m}$. A few elementary considerations
  show that 
  $\E[z_{m_1}z_{n_1}\bar z_{m_2}\bar z_{n_2}]$ is non--zero
  only if $m_1=m_2$ and $n_1=n_2$, or if $m_1=n_2$ and $m_2=n_1$.
  Therefore
  \[
    \begin{multlined}[.9\linewidth]
      \E[|J_k^N(t) - J_k^{N'}(t)|^2]
        = 2\sum_{m+n=k}^{N\leftrightarrow N'}
          (m\cdot n)^2\E[|z_m|^2|z_n|^2] \leq {}\\
        \leq c\sum_{m+n=k}^{N\leftrightarrow N'}
          \frac{(m\cdot n)^2}{|m|^4|n|^4}
          (1-\e^{-2|m|^4t})(1-\e^{-2|n|^4t})
        \leq c\sum_{m+n=k}^{N\leftrightarrow N'}\frac1{|m|^2|n|^2}.
    \end{multlined}
  \]
  The last series above can be estimated with Lemma \ref{l:sumsum},
  indeed since $|m|\vee|n|\geq N$,
  \[
    \sum_{m+n=k}^{N\leftrightarrow N'}\frac1{|m|^2|n|^2}
      \leq \frac{2}{N^\gamma}
        \sum_{\substack{m+n=k\\|n|\leq |m|}}^{N\leftrightarrow N'}
        \frac1{|m|^{2-\gamma}|n|^2}
      \leq \frac{2}{N^\gamma}
        \sum_{m+n=k}^{N\leftrightarrow N'}
        \frac1{|m|^{2-\gamma}|n|^2}
      \leq \frac{c}{N^{\gamma}|k|^{2-\gamma}}.
  \]
  In conclusion,
  \[
    \E[\|\Bt_N(z,z) - \Bt_{N'}(z,z)\|_{H^{-2-\gamma}}^2]
      = \sum_{k\in\Z^2_\star}|k|^{-2\gamma}\E[|J_k^N - J_k^{N'}|^2]
      \leq \frac{c}{N^{\gamma}}\sum_{k\in\Z^2_\star}\frac1{|k|^{2+\gamma}},
  \]
  and the term on the right hand side converges to zero as
  $N,N'\to\infty$.
\end{proof}
\begin{remark}
  In order to define $\Bt$ we have chosen Galerkin projections
  as regularizations of the underlying Wiener--process and passed
  to the limit in order to define the solution. 
  We will see that other regularizations, for example convolution operators, 
  yield exactly the same result.
\end{remark}
We shall need higher moments of $\Bt(z,z)$ for our considerations on
the non--linear problem. We shall derive the claim from
hyper--contractivity of Gaussian measures \cite{Nel1973,Sim1974}.
\begin{proposition}\label{p:nelson}
  Given $\gamma>0$ and $p>1$,
  there is a constant $c>0$ such that 
  \[
    \sup_{t>0}\E\bigl[\|\Bt(z(t),z(t))\|_{H^{-2-\gamma}}^p\bigr]
      \leq c.
  \]
\end{proposition}
\begin{proof}
  For the second moment we can proceed as in the previous lemma,
  using again the elementary estimate of Lemma \ref{l:sumsum},
  \[
    \E[|J_k^N|^2]
      = \sum_{\substack{m+n=k\\|m|,|n|\leq N}}
          (m\cdot n)^2\E[|z_m|^2|z_n|^2]
      \leq c\sum_{m+n=k}\frac1{|m|^2|n|^2}
      \leq \frac{c}{|k|^2}\log(1+|k|).
  \]
  Thus
  \[
    \E\bigl[\|\B_N(z(t),z(t))\|_{H^{-2-\gamma}}^2\bigr]
      \leq \sum_k |k|^{-2\gamma}\E[|J_k^N|^2]
      \leq c\sum_k \frac{\log(1+|k|)}{|k|^{2+2\gamma}}
  \]
  and the second moment is finite. To prove that all moments are finite,
  consider an integer $p\geq1$. Theorem I.22 of \cite{Sim1974} yields
  that
  \[
    \E[|J_k^N|^{2p}]
      \leq (2p-1)^{2p}\bigl(\E[|J_k^N|^2]\bigr)^p,
  \]
  hence by the H\"older inequality,
  \[
  \begin{multlined}[.8\linewidth]
    \E[\|\B_N(z(t),z(t))\|_{H^{-2-\gamma}}^{2p}]
      = \E\Bigl[\Bigl(\sum_k |k|^{-2\gamma}|J_k^N|^2\Bigr)^p\Bigr]\leq{}\\
      \leq \Bigl(\sum_k |k|^{-2\gamma}\bigl(\E[|J_k^N|^{2p}]\bigr)^{\frac1p}\Bigr)^p
      \leq c_p,
  \end{multlined}
  \]
  and the moment of order $2p$ is uniformly bounded in time.
\end{proof}
\section{Proof of Theorem \ref{t:mild}}

Fix $\epsilon\in(0,\tfrac12)$ and $h_0\in\Hz^1$.
Let $\mathcal{T}$ be the operator that takes as its value
the right--hand side of \eqref{e:vmild}, namely
\[
  \T v(t)
    \eqdef S(t)h_0
        - \int_0^t S(t-s)\B(v,v)\,ds
        - \int_0^t S(t-s)\bigl(\Bt(z,z) + 2\B(z,v)\bigr)\,ds.
\]

We use the standard contraction fixed point theorem. To this end we
show that by choosing $\rho,T$ suitably, $\T$ maps $\X(\epsilon,\rho,T)$
into itself. By possibly taking a smaller value of $\rho$, $\T$ is
also a contraction.
\subsubsection*{The self--mapping property}

Set $R_{h_0}(T)\eqdef\sup_{[0,T]}t^{\frac\epsilon4}
\|S(t)h_0\|_{H^{1+\epsilon}}$, then it is easy to show
(\cite{FujKat1964}, see also \cite[Lemma C.1]{BloRom2009})
that $R_{h_0}(T)\longrightarrow0$ as $T\to0$.
Continuity of $\T v$ in $L^2(\torus)$ is standard
(see \cite[Lemma C.1]{BloRom2009}).
Moreover, by Lemma \ref{l:Bsquare}, with
$a\in[1-2\epsilon,1-\epsilon)$, if $t\leq T$,
\[
  \begin{aligned}
    \Bigl\|\int_0^t S(t-s)\B(v,v)\,ds\Bigr\|_{H^{1+\epsilon}}
      &\leq \int_0^t \|A^{\frac14(3+a+\epsilon)}S(t-s)
              A^{-\frac14(a+2)}\B(v,v)\|_{L^2}\,ds\\
      &\leq \int_0^t \frac{c}{(t-s)^{\frac14(3+a+\epsilon)}}
              \|v(s)\|_{H^{1+\epsilon}}^2\,ds\\
      &\leq ct^{-\frac\epsilon4}\|v\|_{\epsilon,T}^2t^{\frac14(1-a-2\epsilon)}.
  \end{aligned}
\]

The mixed term is estimated with the help of Corollary \ref{c:mixed},
with $\gamma<1-\epsilon$, $\alpha=1-(\epsilon\wedge\gamma)/2$ and
$q>4/(\epsilon\wedge\gamma\wedge (1-\epsilon-\gamma))$, and the H\"older inequality,
\[
  \begin{aligned}
    \Bigl\|\int_0^t S(t-s)\B(v,z)\,ds\Bigr\|_{H^{1+\epsilon}}
      &\leq\int_0^t \|A^{\frac14(3+\epsilon+\gamma)}S(t-s)
        A^{-\frac14(2+\gamma)}\B(v,z)\|_{L^2}\,ds\\
      &\leq \|v\|_{\epsilon,T} \int_0^t\frac{c\|z\|_{W^{\alpha,q}}}
        {(t-s)^{\frac14(3+\epsilon+\gamma)}s^{\frac\epsilon4}}\,ds\\
      &\leq ct^{-\frac\epsilon4}T^{e_1}\|v\|_{\epsilon,T}
        \Bigl(\int_0^T\|z\|_{W^{\alpha,q}}^q\,ds\Bigr)^{\frac1q}.
  \end{aligned}
\]
where $e_1 = \tfrac14(1-\epsilon-\gamma)-\tfrac1q$ is positive and the integrals
are well defined by the choice of $q$.

Finally for the quadratic term in $z$ we use Proposition \ref{p:nelson}.
Let $2<\gamma'<3-\epsilon$ and $p'$ such that
$\tfrac14p'(1+\gamma'+\epsilon)<1$, then by the H\"older inequality,
\[
\begin{aligned}
  \Bigl\|\int_0^tS(t-s)\Bt(z,z)\,ds\Bigr\|_{H^{1+\epsilon}}
    &\leq \int_0^t \bigl\|A^{\frac14(1+\epsilon+\gamma')}S(t-s)
      \bigl(A^{-\frac{\gamma'}4}\Bt(z,z)\bigr)\bigr\|_{L^2}\,ds\\
    &\leq \Bigl(\int_0^t(t-s)^{-\frac14(1+\epsilon+\gamma')p'}\,ds\Bigr)^{\frac1{p'}}
      \Bigl(\int_0^t \|\Bt(z,z)\|_{H^{-\gamma'}}^{q'}\,ds\Bigr)^{\frac1{q'}}\\
    &\leq c t^{-\frac\epsilon4} T^{e_2}
      \Bigl(\int_0^T \|\Bt(z,z)\|_{H^{-\gamma'}}^{q'}\,ds\Bigr)^{\frac1{q'}},
\end{aligned}
\]
where $q'$ is the H\"older conjugate exponent of $p'$ and
$e_2 = \tfrac1{p'}-\tfrac14(\gamma'+1)$ is positive.

The three estimates together yield
\begin{equation}\label{e:self}
  \|\T v\|_{\epsilon,T}
    \leq R_{h_0}(T)
      + c_0\rho^2
      + c_0\rho T^{e_1}Z_1(T)
      + c_0 T^{e_2} Z_2(T),
\end{equation}
where $Z_1(T)$ is the norm of $z$ in $L^q(0,T;\Wz^{\alpha,q})$,
and $Z_2(T)$ is the norm of $B(z,z)$ in $L^{q'}(0,T;H^{-\gamma'}_\wp)$.
All the quantities in the displayed formula above converge to $0$
as $T\to0$, so for $T$ small enough we can find a positive value of $\rho$
that satisfies the self--mapping property.
\subsubsection*{The contraction property}

The contraction property follows from similar estimates.
Let $v_1,v_2\in\X(\epsilon,\rho,T)$, then
\[
  \T v_1(t) - \T v_2(t)
    = \int_0^t S(t-s)\B(v_1+v_2,v_2-v_1)\,ds
      + 2\int_0^t S(t-s)\B(z,v_2-v_1).
\]
We use Lemma \ref{l:Bsquare} for the first term and Corollary
\ref{c:mixed} for the second term (with the same choice for
the value of the parameters as the previous part),
\begin{equation}\label{e:contract}
  \begin{aligned}
    \|\T v_1 - \T v_2\|_{\epsilon,T}
      &\leq c\|v_1+v_2\|_{\epsilon,T}\|v_1-v_2\|_{\epsilon,T}
         + cT^{e_1}Z_1(T)\|v_1-v_2\|_{\epsilon,T}\\
      &\leq c_0\bigl(\rho + T^{e_1}Z_1(T)\bigr)\|v_1 - v_2\|_{\epsilon,T}.
  \end{aligned}
\end{equation}
and again by choosing $T$ small enough the mapping is a contraction.

Given $a\in(0,1)$ and $b\in\bigl(0,\smash{\tfrac{a^2}{4c_0}}\bigr)$,
where $c_0$ is the constant appearing in \eqref{e:self} and
\eqref{e:contract}, choose $\rho$ such that $c_0\rho^2-a\rho+b\leq0$.
Let
\[
  \tau_a^c
    \eqdef\inf\{t: c_0t^{e_1}Z_1(t)>1-a\},
      \qquad
  \tau_b^s
    \eqdef\inf\{t: R_{h_0}(t) + c_0t^{e_2}Z_2(t)>b\},
\]
and choose $T<\tau_a^c\wedge\tau_b^s$. With these choices and
positions, it is immediate to verify that the right--hand side
of \eqref{e:self} is smaller or equal than $\rho$ and that the
Lipschitz constant of $\T$ appearing in \eqref{e:contract}
is smaller than $1$. It turns out that
$\tau_{h_0}\geq	\tau_a^c\wedge\tau_b^s$ and, since
$\tau_a^c>0$ and $\tau_b^s>0$ almost surely, the same
holds for $\tau_{h_0}$.
\section{Other regularizations}

Let $\Phi$ be a bounded operator on $\Lz(\torus)$ and consider the
associated stochastic convolution,
\[
  z^\Phi(t)
    = \int_0^t S(t-s)\Phi\,dW(s)
    = \sum_{k\in\Z^2_\star}z_k^{\Phi}(t)e_k,
\] 
where
\[
  z_k^\Phi(t)
    = \int_0^t \e^{-(t-s)|k|^4}\,d\beta_k^\Phi,
      \qquad\text{and}\qquad
  \beta_k^\Phi(t)
    =\langle \Phi W(t), e_k\rangle,
\]
and the $\beta_k^\Phi$ are Brownian motions. These are in general
non independent, unless $\Phi$ is diagonal in the Fourier-basis $e_k$.
We recall that $\Phi$ is regularizing if it satisfies the conclusions
of Lemma \ref{l:admissible}, whose proof is given below.
\begin{proof}[Proof of Lemma \ref{l:admissible}]
  Condition \eqref{e:absconv} ensures that the terms $J_k$ defined
  in \eqref{e:Jk} are {a.s.} absolutely convergent, hence the
  computations in Lemma \ref{l:wick} can be made rigorously for
  $z^\Phi$ without relying on the spectral truncation.
\end{proof}
\subsection{Good approximations}

Here we prove Theorems \ref{thm:or2} and \ref{thm:konsist}.
The proof of Theorem \ref{thm:or2} is a straightforward
modification of Proposition \ref{p:zreg}.
\begin{proof}[Proof of Theorem \ref{thm:or2}]
  Let $\gamma\in(0,1)$ be the value given in the statement of
  the theorem. If $x,y\in\torus$,
  \[
  \begin{aligned}
    \lefteqn{\E\bigl[|z^{\Phi_N}(t,x)-z(t,x)- z^{\Phi_N}(t,y)+z(t,y)|^2\bigr]}\qquad\\ 
      &=\sum_{m,n}\E\bigl[(z_m^{\Phi_N}(t)-z_m(t))
        \overline{(z_n^{\Phi_N}(t)-z_n(t)}\bigr]
        (e_m(x)-e_m(y))(\overline{(e_n(x)-e_n(y))}\\
      &\leq C\sum_{m,n}\frac{|\langle(\Phi_N-I)^\star(\Phi_N-I)e_m,e_n \rangle|}
        {|m|^4+|n|^4}(1\wedge|m(x-y)|)(1\wedge |n(x-y)|)\\
      &\leq C\sum_{m,n}\frac{|\langle(\Phi_N-I)^\star(\Phi_N-I)e_m,e_n\rangle|}
        {(|m|+|n|)^{4-2\gamma}} |x-y|^{2\gamma},
  \end{aligned}
  \]
  where we used that
  \[
    \E\bigl[(\beta_m^{\Phi_N}(1)-\beta_k(1))
        (\overline{\beta_n^{\Phi_N}(1)-\beta_n(1)}\bigr] 
      =\langle (\Phi_N-I)^\star(\Phi_N-I) e_m,e_n \rangle.
  \]
  As in Proposition \ref{p:zreg}, Gaussianity and the definition of
  the norm in $W^{\alpha,p}$ yield
  \[
    \sup_{t\in[0,T]}\E \|z^{\Phi_N}(t)-z(t)\|^p_{W^{\alpha,p}}
      \leq C\Big(\sum_{m,n}\frac{|\langle(\Phi_N-I)^\star(\Phi_N-I)e_m,e_n\rangle|}
        {(|m|+|n|)^{4-2\gamma}}\Big)^{p/2}.
  \]
  The Lebesgue dominated convergence theorem for the double sum on the right
  hand side concludes the proof.
\end{proof}
\begin{proof}[Proof of Theorem \ref{thm:konsist}]
  Consider two different regularizing  operators 
 $\Phi$ and $\Psi$. Of course we have in mind $\Phi_N$ and $\Psi_N$, but we omit 
 the index $N$ in the following. First
  \[
    \B(z^\Phi ,z^\Phi)-\B(z^\Psi ,z^\Psi)=  \B(z^\Phi+z^\Psi ,z^\Phi-z^\Psi)\;.
  \]
  Define  for every $k\in\Z^2$,
  \[
    B_k^\pm = \beta_k^\Phi \pm \beta_k^\Psi = \langle W(t), \Phi e_k \pm \Psi e_k\rangle
    \quad \text{and}\quad
    z_k^\pm = \int_0^t e^{-(t-s)|k|^4}dB^\pm_k \;.
  \]
  Modify moreover the definition of $J_k$
  \[
  \hat{J}_k(t)
    = \sum_{m+n=k} m\cdot n\ z^+_m(t) z^-_n(t).
  \]
  Thus
  \[
    \B(z^\Phi ,z^\Phi) - \B(z^\Psi ,z^\Psi)
      =  \sum_{k\in\Z^2} |k|^2 \hat{J}_k e_k
  \]
  and
  \begin{equation}
    \label{e:so*}
       \E\|\B(z^\Phi ,z^\Phi) - \B(z^\Psi ,z^\Psi)\|^2_{H^{-2-\gamma}}
      = \sum_{k\in\Z^2} |k|^{-2\gamma} \E |\hat{J}_k|^2
  \end{equation}
  By exchanging expectation and summation,
  \[
   \E|\hat{J}_k|^2
     = \sum_{\substack{m_1+n_1=k\\ m_2+n_2=k}} (m_1\cdot n_1)(m_2\cdot n_2)
         \E z^+_{m_1} z^-_{n_1}\overline{z^+_{m_2}}
         \overline{z^-_{n_2}}.
  \]
  Wick's formula \cite[Proposition I.2]{Sim1974} yields 
  \[
    \E z^+_{m_1} z^-_{n_1} \overline{z^+_{m_2}} \overline{z^-_{n_2}}
      = \E z^+_{m_1} z^-_{n_1} \E \overline{z^+_{m_2}} \overline{z^-_{n_2}} 
        +\E z^+_{m_1} \overline{z^+_{m_2}}  \E z^-_{n_1}\overline{z^-_{n_2}} 
        +\E z^+_{m_1} \overline{z^-_{n_2}} \E\overline{z^+_{m_2}} z^-_{n_1}.
  \]
  Hence, (using the symmetry of variables $n_2 \leftrightarrow m_2$ in the
  last term)
  \begin{equation}\label{e:nondiJ}
    \begin{split}
      \E |\hat{J}_k|^2 
        &=  \Big| \sum_{m+n=k} (m\cdot n)\E z^+_{m} z^-_{n}\Big|^2 \\
        &\quad+  \sum_{\substack{m_1+n_1=k\\ m_2+n_2=k}} (m_1\cdot n_1)(m_2\cdot n_2) \E z^+_{m_1} \overline{z^+_{m_2}} \E z^-_{n_1}\overline{z^-_{n_2}} \\
        &\quad+  \sum_{\substack{m_1+n_1=k\\ m_2+n_2=k}} (m_1\cdot n_1)(m_2\cdot n_2) \E z^+_{m_1} \overline{z^-_{m_2}} \E\overline{z^+_{n_2}} z^-_{n_1}.
    \end{split}
  \end{equation}
  Since $\E\langle W(t),u\rangle\overline{\langle W(t),v\rangle}
  =\langle v,u \rangle$, 
  \[
    \E z^\pm_{m} \overline{z^\mp_{\ell}}
      = \int_0^te^{-(t-s)(|m|^4+|\ell|^4)}\,ds
        \langle\Phi e_\ell \mp \Psi e_\ell,\Phi e_m\pm\Psi e_m\rangle.
  \]
  Hence,
  \[
    |\E z^\pm_{m} \overline{z^\mp_{\ell}}|
      \leq  \frac1{|m|^4+|\ell|^4}
        \big|\langle \Phi e_\ell\mp\Psi e_\ell, \Phi e_m\pm\Psi e_m\rangle\big|,
  \]
  and similarly for other combinations of signs.

  Let us now consider again the sequences $\Phi_N$ and $\Psi_N$.
  We treat the diagonal terms with 
  $m_1=m_2$ and $n_1=n_2$ and the off--diagonal terms differently.
  We have to assume some uniform summability of the off-diagonal terms. 
  This is ensured by the bounds \eqref{e:Phibound1} and \eqref{e:Phibound2}.
  Moreover, all summands go to $0$ in \eqref{e:nondiJ} due to
  the convergence in \eqref{e:Phiconv}. 
  Thus from \eqref{e:so*} by the Lebesgue dominated convergence theorem:
  \[
    \sup_{t\ge0} \E\|\B(z^{\Phi_N} ,z^{\Phi_N})-\B(z^{\Psi_N} ,z^{\Psi_N})\|^2_{H^{-2-\gamma}}
      \to 0 \qquad\text{for }N\to 0.
  \]
  As before, using hyper--contractivity as in the proof 
  of Proposition \ref{p:nelson}, we can show 
  that this holds for all moments and not only for the second. 
\end{proof}
Let us come back to the examples given in Section \ref{s:mild1}.
\begin{example}[Example \ref{exp:or1} resumed]\label{exp:or1bis}
  The operators $\Phi_N$ are diagonal and self--adjoint,
  denote by $\phi_k^N$ the eigenvalues of $\Phi_N$.
  These numbers are (up to a constant) determined by the
  Fourier coefficients of $N^2q_N$. Write
  \[
    N^2q_N(z)
      = \sum_{m\in\Z^2} q_m^N e_m(z),
  \]
  then $\Phi_Ne_k = q_{-k}^N$, thus $\phi_k^N=q_{-k}^N$.
  It is easy to check now that \eqref{e:absconv} is a decay
  condition on the eigenvalues and a sufficient condition
  is given by $|\phi_k^N|\lesssim |k|^{-\eta}$ for some $\eta>0$,
  that is $q$ belongs to $H^\eta$. The off--diagonal assumptions
  are clearly verified. It remains to check the convergence
  \eqref{e:Phiconv} when the $\Phi_N$ are combined with the
  Galerkin truncation operators $\pi_N$. To this end, it is
  sufficient to show that $\phi_k^N\to1$ as $N\to\infty$.
  This can be checked using the fact that $q$ is supported
  around $0$,
  \[
   q_k^N
     = N^2\int_{[-\pi,\pi]^2} q_N(x)e_{-k}(x)\,dx \\ 
     = \int_{[-\pi,\pi]^2} q(z) e_{-k/N}(z)\,dz
     \longrightarrow 1.
  \]
\end{example}
\begin{example}[Example \ref{exp:or2} resumed]\label{exp:or2bis}
  For simplicity of notation extend the operator to complex valued functions by 
  \[
    \Phi_N f(x)
      = \int_{[0,2\pi]^2} {N^2} q_N(x,y)) \overline{f(y)}\,dy =\langle N^2q_N(x,\cdot) , f\rangle\;.
  \]
  Recall that $q_N$ is given by a non-negative smooth $q$ 
  supported in a small neighbourhood of the diagonal $x=y$,
  such that $q_N$ is a periodic extension of $q(Nx,Ny)$ on $\torus\times\torus$.
Denote by $q_{k,\ell}^{(N)}$ the Fourier coefficients of $N^2q_N$,
i.e.
\[
N^2q_N(x,y) = \sum_{k,\ell\in\Z^2} q_{k,\ell}^{(N)} e_k(x)e_\ell(y)\;.
\]
This immediately implies, that 
\[
\langle \Phi_Ne_m , e_n \rangle = \langle \sum_{k \in\Z^2} q_{k,m }^{(N)} e_k , e_n\rangle = q_{n,m }^{(N)}
\]
 and 
  \[
\langle \Phi_Ne_m , \Phi_N e_n \rangle = \langle \sum_{k \in\Z^2} q_{k,m }^{(N)} e_k , \sum_{k \in\Z^2} q_{k,n }^{(N)} e_k\rangle 
= \sum_{k \in\Z^2}   q_{k,m }^{(N)}   q_{k,n }^{(N)} \;.
\]
We can now again check all assumptions of Lemma \ref{l:admissible} 
and  Theorem \ref{thm:or2} and \ref{thm:konsist}.
First \eqref{e:absconv} is true, once $q$ is sufficiently smooth, for
example in some $H^\epsilon$, as this implies that
$q_{k,m }^{(N)} \leq \|N^2 q_N\|_{H^\epsilon} (|k|^2+|\ell|^2)^{\epsilon/2}$
by Lemma \ref{l:sumsum}.

For the next steps, as before, 
let $\Psi^{(N)}=\pi_N$ be the projection onto the first Fourier modes. 
Now,
\[ c_{m,n}  \leq \sup_{N\in\N}\{ |\sum_{k \in\Z^2}  q_{k,m }^{(N)}   q_{k,n }^{(N)}| + \delta_{m,n} + q_{n,m }^{(N)} \}.
\]  
The bounds in (\ref{e:Phibound1}) and  (\ref{e:Phibound2})  are easy to establish, 
as we can verify that $c_{m,n}$ is uniformly bounded
together with the fact that, by Lemma \ref{l:sumsum},  $\sum_{m+n=k}|m|^{-3}|n|^{-3} \leq C(1+|k|)^{-3}$.
In order to establish uniform bounds on $c_{m,n}$, consider 
\[ q_{n,m }^{(N)} = N^2 \int_{\torus}  \int_{\torus} q_N(x,y) e_{-n}(x)e_{-m}(y)\,dx\,dy,\]
and, as $\langle q_N(x,\cdot, e_m \rangle = \sum_{\ell\in\Z^2} q_{\ell,m}e_k(x) $,
\[\sum_{k \in\Z^2} q_{k,m}^{(N)} q_{k,n}^{(N)}
= N^4 \int_{\torus} \int_{\torus} q_N(x,y)q_N(x,z) e_{-m}(x)e_{-n}(y)\,dx\,dy.
\]
The bounds now follow immediately from substituting $N$ and bounds on the support of $q$.
With some more effort, one can also verify that
\[q_{n,m }^{(N)} =   N^{-2} \int_{N\torus}  \int_{N\torus} q(x,y) e_{-n/N}(x)e_{-m/N}(y) dx dy  \to \delta_{n,m}\qquad\text{for}\quad N\to \infty\;.
\]
We rely on a splitting on $N\torus$ into in $N$ translated copies of  $\torus$ here.
The convergence of $\sum_{k \in\Z^2}  q_{k,m }^{(N)}   q_{k,n }^{(N)}$ 
to a Kronecker-Delta is more involved, and we skip details here.
The crucial condition is (\ref{e:thmor2}), 
which does not seem to hold for arbitrary kernel $q$.
First let us remark, that main problem in (\ref{e:thmor2}) are the terms
with $m\not=n$,
as the diagonal terms are easily summable under our assumption. 

A weak sufficient condition for the off diagonal term  
would be to assume that for some $\xi>1$,  
\[
| q_{k,\ell}^{(N)} | \leq C (2+|k|^2+|\ell|^2)^{-\xi}.
\] 
Now we can verify by comparison with integrals, that 
\[
|\langle (\Phi_N-I)e_n, (\Phi_N-I) e_m \rangle |
\leq  \sum_{k \in\Z^2}  |q_{k,m }^{(N)}   q_{k,n }^{(N)}| +  |q_{m,n }^{(N)}| + \delta_{m,n}
\leq C (1+|m|+|n|)^{-2\xi}\;.
\]
Thus (\ref{e:thmor2}) is true, as long as $\gamma<\alpha$.
\end{example}
\subsection{Proof of the stability theorem}\label{sec:stab}

In this section we prove Theorem \ref{thm:stability}.
Let $h_0\in\Hz^1(\torus)$, $\epsilon\in(0,\frac12)$
and let $(\Phi_N)_{N\in\N}$ be a sequence of regularizing
operators satisfying the convergence property \eqref{e:convstab}.
In the following the index $N$ is omitted, and we use a general regularizing
operator $\Phi$ first.

\subsubsection*{Step 1: Existence}

Fix $\epsilon \in(0,\frac12)$ from the proof of existence of solution
in Theorem \ref{t:mild}
and let $v=h-z$ be the solution from Theorem \ref{t:mild}.
We first establish:
\begin{theorem}
Assume that the regularized operator $\Phi$ is such that $z^\Phi\in  C^0([0,1],\Hz^{1+\epsilon})$. 
Then the equation with regularized noise $z^\Phi$ has a unique local solution $v^\Phi$ in $\X(\epsilon,\rho,T)$
for some small random $T>0$.
\end{theorem}

This is follows  from \cite{BloRom2012}
or analogous to the result presented in this paper in Theorems \ref{t:mild}, \ref{t:mild2}, or \ref{t:mild3}.
 
It is straightforward to verify that $v^\Phi - S(t)h_0$ 
and $v-S(t)h_0$ are both continuous with values in $\Hz^{1+\epsilon}$
locally close to $0$. Moreover, we can continue all $v^\Phi$
by standard arguments as an $\Hz^{1+\epsilon}$-valued continuous function, 
until they blow up in $H^{1+\epsilon}$.

Recall for a given large radius $R>0$
\[\tau^R=\inf\{t>0:\ \|v(t)-S(t)h_0\|_{H^{1+\epsilon}} > R  \} 
\]
Moreover, we can define $\tau^\Phi>0$ 
as the maximal time of existence in $H^{1+\epsilon}$,
at which $v^\Phi$ blows up.

\subsubsection*{Step 2: Bounding the error} 

We can define for all $t\in[0,\tau^\Phi\wedge\tau^R)$ the error 
\[
d^\Phi(t) = v^\Phi(t) -v(t)\;.
\]
Define the stopping time, where the error exceeds $1$:
\[
 \tau^\star=\inf\{t>0:\ \|d^\Phi\|_{H^{1+\epsilon}}>1  \} \wedge \tau^R  \wedge 1 
\]
Obviously, $\tau^\Phi \geq \tau^\star>0$.
Using It\^o-formula, we have
\[
\begin{split}
d^\Phi(t)= \int_0^tS(t-s)[& \B(v^\Phi,v^\Phi)-\B(v ,v )
+ \B(v^\Phi,z^\Phi)-\B(v ,z ) 
+\B(z^\Phi,z^\Phi)-\tilde\B(z,z)]ds\\
= \int_0^t S(t-s)[& \B(d^\Phi,d^\Phi)+2\B(v , d^\Phi ) 
+  \B(d^\Phi,z) \\& +  \B(d^\Phi,z^\Phi-z) +  \B(v,z^\Phi-z) 
+\B(z^\Phi,z^\Phi)-\tilde\B(z,z)]ds\\
\end{split}
\]
Here, we rewrote all terms in a way that they depend only on $v$, $z$, $d^\Phi$,
$z-z^\Phi$, and $v-v^\Phi$.
Thus, using the bounds from the proof of Theorem \ref{t:mild}
with
\[\gamma=\epsilon,\quad   
\alpha=\frac12-\epsilon, \quad
q>\frac4\epsilon, \quad 
\beta \in(2,3-\epsilon)  \quad\text{and}\quad 
q' > 4/(3-\beta-\epsilon) 
\]
yields
\[
\begin{split}
\|d^\Phi(t)\|_{H^{1+\epsilon}}
\leq &C\int_0^t (t-s)^{-\frac14(4-\epsilon)}  (\|d^\Phi\|_{H^{1+\epsilon}}^2+\|v\|_{H^{1+\epsilon}} \|d^\Phi\|_{H^{1+\epsilon}} ) ds \\
&+C\int_0^t (t-s)^{-\frac14(3+2\epsilon)} \|d^\Phi\|_{H^{1+\epsilon}} (\|z\|_{W^{\alpha,q}}+\|z^\Phi-z \|_{W^{\alpha,q}})ds\\
&+ C\int_0^t (t-s)^{-\frac14(3+2\epsilon)} \|v\|_{H^{1+\epsilon}}\|z^\Phi-z \|_{W^{\alpha,q}}ds \\
&  + C \|\B(z^\Phi,z^\Phi)-\tilde\B(z,z)\|_{L^{q'}([0,1],H^{-\beta})}\;.
\end{split}
\]
The definition of $\tau^\star$ and using $t\in[0,1]$ yields 
\[
\begin{split}
\|d^\Phi(t)\|_{H^{1+\epsilon}}
\leq &C\int_0^t (t-s)^{-\frac14(4-\epsilon)}  \|d^\Phi\|_{H^{1+\epsilon}}(1+\|v\|_{H^{1+\epsilon}} + \|z\|_{W^{\alpha,q}} + \|z^\Phi-z \|_{W^{\alpha,q}} ) ds \\
&+ C\int_0^t (t-s)^{-\frac14(3+2\epsilon)} \|v\|_{H^{1+\epsilon}}\|z^\Phi-z \|_{W^{\alpha,q}}ds \\
&  + C \|\B(z^\Phi,z^\Phi)-\tilde\B(z,z)\|_{L^{q'}([0,1],H^{-\beta})}\;.
\end{split}
\]
Now define the random variable of the error
\[ 
\mathcal{E}^\Phi =  \|z^\Phi-z \|_{L^p([0,1],W^{\alpha,q})}
+\|\B(z^\Phi,z^\Phi)-\tilde\B(z,z)\|_{L^{q'}([0,1],H^{-\beta})}  \;.
\]
This simplifies the previous estimate for $t\in[0,\tau^\star)$ to

\[
\|d^\Phi(t)\|_{H^{1+\epsilon}}
\leq C\int_0^t (t-s)^{-\frac14(4-\epsilon)}  
\|d^\Phi\|_{H^{1+\epsilon}} (1+R  
+ \|z\|_{W^{\alpha,p}} + C \mathcal{E}^\Phi ) ds + C \mathcal{E}^\Phi\;.
\]

\subsubsection*{Step 3:  Estimates  in probability} 

Now define 
for $\delta\in(0,1)$ the set 
\[
\Omega_{\Phi,\delta} =\{  \mathcal{E}_\Phi \leq \delta  \}
\]
which is a large set, in case if $z^\phi$ approximates $z$ well, 
as $\Prob(\Omega_{\Phi,\delta})\geq 1 - \E\mathcal{E}_\Phi /\delta $.

Thus on the set $\Omega_{\Phi,\delta} $ for $t\leq  \tau^\star$ we obtain
\[
\|d^\Phi(t)\|_{H^{1+\epsilon}}
\leq C\int_0^t (t-s)^{-\frac14(4-\epsilon)}  \|d^\Phi\|_{H^{1+\epsilon}} \cdot ds (1 + \sup_{[0,1]}\|z\|_{W^{\alpha,p}}) + C \delta 
\]
Using Gronwall's inequality 
for $\frac1\delta \|d\|_{H^{1+\epsilon}}$  in the version of \cite{DH81}
yields  the existence of a finite random 
constant $C(\omega)$ independent of $\delta$ such that
\[
\sup_{t\in[0,\tau^\star]}\|d^\Phi(t)\|_{H^{1+\epsilon}} \leq C(\omega) \delta 
\]
Hence, we verified that for all $\delta_0 \in(0,1)$ on the set
$\{ C(\omega) < \delta_0/\delta  \}\cap \Omega_{\Phi,\delta}$ 
we have $\tau^\star \geq 1\wedge\tau_R$ and $ \|d^\Phi\|_{H^{1+\epsilon}} \leq \delta_0$.
Thus 
\[
\begin{split}
\Prob( \sup_{[0,1\wedge\tau_R]} \|d^\Phi\|_{H^{1+\epsilon}} > \delta_0)  
& \leq  \Prob( \{ C(\omega) \geq \delta_0/\delta\}\cup \Omega_{\Phi,\delta}^c) \\
& \leq  \Prob( \{ C(\omega) \geq  \delta_0/\delta\}) + \E\mathcal{E}^\Phi /\delta 
\end{split}
\]
Fixing $\delta= \sqrt{\E\mathcal{E}^\Phi}$  yields
\[\Prob( \sup_{[0,1\wedge\tau_R]} \|d^\Phi\|_{H^{1+\epsilon}} \geq \delta_0) 
\leq   \Prob( \{ C(\omega) \geq  \delta_0 / \sqrt{\E\mathcal{E}^\Phi}\}) 
+ \sqrt{\E\mathcal{E}^\Phi}
\]
\subsubsection*{Step 4:  Convergence} 

Now we can apply the results of the previous step to our sequence $\Phi_N$. Here $\E\mathcal{E}^{\Phi_N} \to 0$ for $N\to \infty$.
Thus $\sup_{[0,1\wedge\tau_R]} \|d^\Phi\|_{H^{1+\epsilon}} \to 0$ 
in probability.

\section{Rougher noise}

In this section we deal with a bounded linear operator
$\Phi$ on $\Lz^2(\torus)$ and we assume that
\begin{itemize}
  \item $\Phi\e_k = \phi_k e_k$ for every $k\in\Z^2_\star$,
  \item there is $\beta>0$ such that for every $k\in\Z^2_\star$,
    \[
      |\phi_k|^2
        \leq c|k|^\beta.
    \]
\end{itemize}
Since $\Phi$ is real valued, we clearly have that
$\bar\phi_k = \phi_{-k}$.
\begin{remark}\label{r:optimal}
  It is obvious that any additional information on the
  $\phi_k$ would in principle improve the results of
  this section. On the other hand our results are
  optimal once we know that $|\phi_k|^2\approx |k|^{\beta}$,
  namely there are $c,c'>0$ such that 
  $c'|k|^\beta\leq|\phi_k|^2\leq c|k|^\beta$.
\end{remark}

We wish to find a mild solution of the following equation,
\begin{equation}\label{e:hphi}
  dh + \Delta^2 h + \B(h,h) = \Phi\,dW,
\end{equation}
where $W$ is a cylindrical Wiener process on $L^2(\torus)$,
again as $h = v + z$, where
\begin{equation}\label{e:zphi}
  z(t)
    = \int_0^tS(t-s)\Phi\,dW
    = \sum_{k\in\Z^2_\star}\phi_k z_k e_k,
\end{equation}
and the $z_k$ are defined as in \eqref{e:zkappa}.
The following lemma can be easily proved as in
Proposition \ref{p:zreg}.
\begin{lemma}
  Let $\beta<2$ and let $z$ be the process defined in \eqref{e:zphi}.  
  For every $p\geq1$ and $s\in\bigl(0,1-\tfrac\beta2\bigr)$, 
  \[
    \sup_{t>0} \E\bigl[\|z(t)\|^p_{W^{s,p}}\bigr]
      <\infty.
  \]
\end{lemma}
We turn to the definition of $\B(z,z)$. The counterpart
of Lemma \ref{l:wick} and Proposition \ref{p:nelson} is
the following proposition.
\begin{proposition}
  Let $\beta<1$ and let $z$ be the stochastic convolution
  defined in \eqref{e:zphi}. Then $(\B_N(z,z))_{N\geq1}$
  is a Cauchy sequence in $L^2(\Omega;\Hz^{-2-\gamma})$ for every
  $\gamma>\beta$.
  
  Denote by $\Bt(z,z)$ the limit of this sequence,
  which is well--defined as an element of $\Hz^{-2-\gamma}$,
  for $\gamma>\beta$. Given $\gamma>\beta$ and $p\geq1$,
  there is a constant $c>0$ such that 
  \[
    \sup_{t>0}\E\bigl[\|\Bt(z(t),z(t))\|_{H^{-2-\gamma}}^p\bigr]
      \leq c.
  \]
\end{proposition}
The proof of the above proposition follows the same lines of the
above mentioned results (direct computations and hyper--contractivity).
\begin{remark}\label{r:threshold}
  The restriction $\beta<1$ in the assumptions in the
  above proposition is necessary, at least in our simple
  approach. Denote by $J_k$ the term
  $J_k = \sum_{m+n=k} (m\cdot n)\phi_m\phi_n z_m z_n$,
  and assume, as in Remark \ref{r:optimal}, that
  $|\phi_k|^2\approx|k|^\beta$. Then it is easy to see that
  \[
    \E[|J_k|^2]
      \approx \sum_{m+n=k}\frac1{|m|^{2-\beta}|n|^{2-\beta}},
  \]
  which converges only if $\beta<1$.
  
  In order to manage the problem for $\beta\geq1$,
  it is necessary to consider another renormalization
  of the non--linearity. While the first renormalization
  has been hidden by the Laplace term in front of the
  squared gradient (killing the infinite constant,
  see Lemma 3.1), it looks like one should need a
  more refined approach as in \cite{Hai2013,GubImkPer2012}
  to proceed further, but this does not fit in our
  simple approach.
\end{remark}
We have all ingredients to prove the existence of a local
mild solution of \ref{e:hphi}, interpreted as $h = v + z$.
Here $v$ is the solution of the mild formulation \eqref{e:vmild},
and $z$ is given by \eqref{e:zphi}.
\begin{proof}[Proof of Theorem \ref{t:mild2}]
  Let us look first at the ``self--mapping property''. We need to estimate
  the three terms $\mathcal{I}_1 = \int_0^tS(t-s)\B(v,v)\,ds$,
  $\mathcal{I}_2 = \int_0^tS(t-s)\B(v,z)\,ds$
  and $\mathcal{I}_3 = \int_0^tS(t-s)\Bt(z,z)\,ds$.

  The estimate of $\mathcal{I}_1$ is the same of Theorem
  \ref{t:mild}. For $\mathcal{I}_2$ we use Corollary
  \ref{c:mixed}. Choose $\gamma$ such that
  $\tfrac\beta2<\gamma<1-\epsilon$, $\alpha$ such that
  $\tfrac\beta2<1-\alpha<\epsilon\wedge\gamma$,
  and $q$ large enough, then
  \[
  \|\mathcal{I}_2\|_{H^{1+\epsilon}}
    \leq c\int_0^t \|A^{\frac{3+\epsilon+\gamma}4}S(t-s)
      A^{-\frac{\gamma+2}4}\B(v,z)\|_{L^2}\,ds
    \leq ct^{-\frac\epsilon4}T^{e_1}
      \Bigl(\int_0^T\|z\|_{W^{\alpha,q}}^q\Bigr)^{\frac1q}.
  \]
  Finally, the estimate of $\mathcal{I}_3$ is the same of Theorem
  \ref{t:mild} and for $\beta<\gamma'<1-\epsilon$ and $q$ large
  enough,
  \[
    \|\mathcal{I}_3\|_{H^{1+\epsilon}}
      \leq c\int_0^t\|A^{\frac{3+\epsilon+\gamma'}4}S(t-s)
        A^{-\frac{\gamma'+2}4}\Bt(z,z)\|_{L^2}\,ds
      \leq ct^{-\frac\epsilon4}T^{e_2}
        \|\Bt(z,z)\|_{L^{q'}(H^{-2-\gamma'})}.
  \]
  The contraction property follows by the same inequalities,
  as in Theorem \ref{t:mild}.
\end{proof}
\begin{remark}
  The ``troublemaker'' in the proof of the above theorem
  is the last term, the one denoted by $\mathcal{I}_3$.
  Indeed, the computations for $\mathcal{I}_1$ and
  $\mathcal{I}_2$ work for any $\beta\in(0,1)$,
  given $\epsilon\in(0,\tfrac12)$ for the first
  term, and $\epsilon\in(\tfrac\beta2,1-\tfrac\beta2)$
  for the second term. The term $\mathcal{I}_3$ requires
  $\epsilon\in(0,1-\beta)$, hence $\beta<\tfrac23$.
\end{remark}
\subsection{The second order expansion in Wiener chaos}

Assume now $\beta\in[\tfrac23,1)$ and consider
the decomposition $h = u + \zeta + z$ of $h$, where
where $\zeta$ solves
\[
  \dot\zeta + A\zeta + \Bt(z,z) = 0,
    \qquad
  \zeta(0) = 0,
\]
and $u$ is the mild solution of
\begin{equation}\label{e:higher}
  \dot u + Au + \B(u,u) + 2\B(u,z) + 2\B(u,\zeta)
     + 2\Bt(\zeta,z) + \B(\zeta,\zeta) = 0,
\end{equation}
with initial condition $u(0) = h(0)$, and
$\Bt(\zeta,z)$ is defined below in Lemma \ref{l:addreg2}.
We can write $\zeta$ as
\begin{equation}\label{e:zeta}
  \zeta(t)
    = - \int_0^t S(t-s)\Bt(z,z)\,ds,
\end{equation}
then by maximal regularity, $\zeta\in\Hz^{1+a}$, for
every $a<1-\beta$. Notice that there is no additional
gain in regularity if we try
a direct computation in the style of Lemma \ref{l:wick}.
Roughly speaking, we have already ``used'' the effect of
cancellations in the definition of $\Bt(z,z)$.

Although we have gained additional regularity
for the term $\B(\zeta,\zeta)$ appearing in the equation
for the remainder $u$, this is not enough. Indeed
a standard multiplication theorem in Sobolev spaces
(Lemma \ref{l:Bsquare}) yields that
$\B(\zeta,\zeta)\in H_\wp^{-2-\gamma}$
for $\gamma>2\beta-1$, which by a quick computation
allows, in the fixed point argument, for
$\epsilon\in(\tfrac\beta2,2-2\beta)$,
thus $\beta<\tfrac45$. Even worse, when dealing
with $\B(\zeta,z)$, we see that we still have not
enough regularity to give a meaning to this term.
We shall solve both problems exploiting cancellations.

Before proceeding, we state a few preliminary remarks.
First, we know that $z = \sum_k \phi_k z_k e_k$ and that
\[
  \zeta(t)
    = \int_0^t S(t-s)\Bt(z,z)\,ds
    = \lim_N \int_0^t S(t-s)B(z^N,z^N)\,ds
    = \lim_N \zeta^N(t),
\]
with obvious definition of $\zeta_N$. If we write
$\phi_k^N = \phi_k$ if $|k|\leq N$ and $\phi_k^N=0$
otherwise, we have
$z^N(t) = \sum_k \phi_k^N z_k(t) e_k$, and
$\zeta^N(t) = \sum_{k\in\Z^2_\star}|k|^2\J_k^N(t)e_k$
where we have set for every $k\in\Z^2_\star$,
\[
  J_k^N
    = \sum_{m+n=k}(m\cdot n)\phi_m^N\phi_n^Nz_mz_n
      \qquad\text{and}\qquad
  \mathcal{J}_k^N(t)
    = \int_0^t\e^{-|k|^4(t-s)}J_k^N(s)\,ds.
\]
Finally, we remark a simple computation that will be
useful in the next sections. Let $m,n\in\Z^2_\star$, then
$\E[z_m(t)z_n(s)]=0$ unless $m+n=0$. In the latter case,
\begin{equation}\label{e:simple}
  \E[z_m(t)z_n(s)]
    = \E[z_m(t)\bar z_m(s)]
    = \frac1{|m|^4}\bigl(\e^{-|m|^4|t-s|} - \e^{-|m|^4(t+s)}\bigr)
    \leq \frac1{|m|^4}.
\end{equation}
\subsubsection{The term $\B(\zeta,\zeta)$}

We shall prove the following result.
\begin{lemma}\label{l:addreg1}
  Let $\beta\in(0,1)$. Then for every $\gamma>\beta-1$
  and $p\geq1$,
  \[
    \sup_{t\geq0}\E[\|\B(\zeta,\zeta)\|_{H^{-2-\gamma}}^p]
      <\infty.
  \]
\end{lemma}
Since $\B(\zeta,\zeta) = \lim_N\B(\zeta^N,\zeta^N)$
in $H^{-2-\gamma}_\wp$ for every $\gamma>2\beta-1$,
to prove additional regularity of $\B(\zeta,\zeta)$ it is
enough to prove that $(\B(\zeta^N,\zeta^N))_{N\geq1}$
is uniformly bounded in $L^2(\Omega;H^{-2-\gamma}_\wp)$ (and hence in
$L^p(\Omega)$ for every $p\geq1$ by hyper--contractivity, see
Proposition \ref{p:nelson}) for every $\gamma>\beta-1$.

We have
\[
  \B(\zeta^N,\zeta^N)
    = \sum_{k\in\Z^2_\star}|k|^2\Bigl(
      \sum_{h+\ell=k}|\ell|^2|h|^2(h\cdot\ell)\J_h^N\J_\ell^N\Bigr)e_k,
\]
hence we can estimate its norm, to obtain
\[
\begin{aligned}
  \E[\|\B(\zeta^N,\zeta^N)\|_{H^{-2-\gamma}}^2]
    &= \sum_{k\in\Z^2_\star}|k|^{-2\gamma}\E\Bigl|
       \sum_{h+\ell=k}|\ell|^2|h|^2(h\cdot\ell)\J_h^N\J_\ell^N\Bigr|^2\\
    &\leq \sum_{k\in\Z^2_\star}|k|^{-2\gamma}
      \sum_{\substack{h_1+\ell_1=k\\h_2+\ell_2=k}}
      \frac{\sup\bigl|\E[J_{h_1}^NJ_{\ell_1}^N
      \bar J_{h_2}^N\bar J_{\ell_2}^N]\bigr|}
      {|h_1|\,|h_2|\,|\ell_1|\,|\ell_2|}.
\end{aligned}
\]
Each $\E[J_{h_1}^NJ_{\ell_1}^N\bar J_{h_2}^N\bar J_{\ell_2}^N]$
contains sums of terms like
\begin{equation}\label{e:eight}
  \E[z_{m_1}(r_1)z_{n_1}(r_1)z_{a_1}(s_1)z_{b_1}(s_1)\bar z_{m_2}(r_2)
    \bar z_{n_2}(r_2)\bar z_{a_2}(s_2)\bar z_{b_2}(s_2)],
\end{equation}
with $m_i+n_i=h_i$, $a_i+b_i=\ell_i$, $i=1,2$. Wick's formula yields
$105$ products of four expectations. Of these terms, $45$ are zero
by \eqref{e:simple}, since they contain terms like $\E[z_{m}z_{n}]$
with $m+n\neq0$. By symmetry, we can
collect the remaining terms in four classes as suggested in the picture%
\footnote{A dashed line means that the sum of the two connected labels is one
of the numbers $h_i,\ell_i$, the continuous line means that the two labels
are in the product of the same expectation, for instance the first
picture reads $\E[z_{m_1}z_{a_1}]\E[z_{n_1}z_{b_1}]
\E[\bar z_{m_2}\bar z_{a_2}]\E[\bar z_{n_2}\bar z_{b_2}]$.}
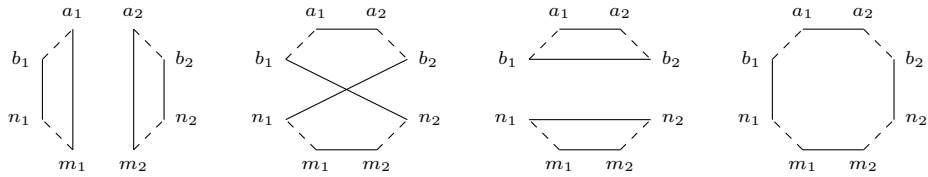
\begin{figure}[h]
  \begin{tikzpicture}[x=4mm,y=4mm]
    \def\basics#1{
      \begin{scope}[shift={(#1,0)}]
        \draw[dashed] (1,0)--(0,1) (0,3)--(1,4) (3,4)--(4,3) (3,0)--(4,1);
        \node[anchor=north] (m1) at (1,0) {\Tiny $m_1$};
        \node[anchor=east] (n1) at (0,1) {\Tiny $n_1$};
        \node[anchor=east] (b1) at (0,3) {\Tiny $b_1$};
        \node[anchor=south] (a1) at (1,4) {\Tiny $a_1$};
        \node[anchor=north] (m2) at (3,0) {\Tiny $m_2$};
        \node[anchor=west] (n2) at (4,1) {\Tiny $n_2$};
        \node[anchor=south] (a2) at (3,4) {\Tiny $a_2$};
        \node[anchor=west] (b2) at (4,3) {\Tiny $b_2$};
      \end{scope}
    }
    \basics{0}
    \draw (1,0)--(1,4) (0,1)--(0,3) (3,0)--(3,4) (4,1)--(4,3); 
    \basics{8}
    \draw[shift={(8,0)}] (0,3)--(4,1) (1,0)--(3,0) (0,1)--(4,3) (1,4)--(3,4); 
    \basics{16}
    \draw[shift={(16,0)}] (3,4)--(1,4) (0,3)--(4,3) (0,1)--(4,1) (1,0)--(3,0); 
    \basics{24}
    \draw[shift={(24,0)}] (0,3)--(0,1) (1,0)--(3,0) (4,1)--(4,3) (1,4)--(3,4); 
  \end{tikzpicture}
  \caption{Graphical representation of classes of terms \eqref{e:eight}}
\end{figure}
The first term gives no contribution, since it is non-zero only if
$m_1+a_1=0$ and $b_1+n_1=0$, hence $k = h_1+\ell_1=
m_1+n_1+a_1+b_1=0$. The second term contains terms like
\[
  \bigl|\E[z_{m_1}\bar z_{m_2}]\E[z_{n_1}\bar z_{b_2}]\E[z_{a_1}\bar z_{a_2}]
      \E[z_{b_1}\bar z_{n_2}]\bigr|
    \leq\frac{\delta_{m_1-m_2}\delta_{n_1-b_2}\delta_{a_1-a_2}
      \delta_{b_1-n_2}}{|m_1|^4|n_1|^4|a_1|^4|b_1|^4},
\]
hence $n_1=b_2=h_1-m_1$, $b_1=n_2=h_2-m_1$, $a_1=a_2=\ell_1-h_2+m_1$
and for these indices,
\[
\begin{multlined}
  \smashoperator{\sum_{\substack{h_1+\ell_1=k\\h_2+\ell_2=k}}}
      \frac{\sup\bigl|\E[J_{h_1}^NJ_{\ell_1}^N
      \bar J_{h_2}^N\bar J_{\ell_2}^N]\bigr|}
      {|h_1|\,|h_2|\,|\ell_1|\,|\ell_2|}
    \leq\smashoperator[l]{\sum_{\substack{h_1+\ell_1=k\\h_2+\ell_2=k}}}
      \frac{c}{|h_1|\,|h_2|\,|\ell_1|\,|\ell_2|}\cdot\\
      \qquad\qquad\cdot\Bigl(\sum_{m_1}\frac{c}
      {|m_1|^{2-\beta}|h_2-m_1|^{2-\beta}|h_1-m_1|^{2-\beta}
      |\ell_2-h_1+m_1|^{2-\beta}}\Bigr).
\end{multlined}
\]
Change the order of sums, summing first over $h_1$, then
over $h_2$ and finally over $m_1$, then by the Cauchy--Schwartz
inequality and Lemma \ref{l:sumsum},
\[
\begin{multlined}
  \sum_{h_1+\ell_1=k}
      \frac{c}{|h_1|\,|\ell_1|\,|h_1-m_1|^{2-\beta}
      |\ell_2-h_1+m_1|^{2-\beta}}\leq{}\\
    \leq \Bigl(\sum_{h_1}\frac1{|h_1|^2|\ell_1|^2}\Bigr)^{\frac12}
      \Bigl(\sum_{h_1}\frac1{|h_1-m_1|^{4-2\beta}
      |\ell_2-h_1+m_1|^{4-2\beta}}\Bigr)^{\frac12}
    \leq \frac{\sqrt{\log(1+|k|)}}{|k|\,|\ell_2|^{2-\beta}}.
\end{multlined}
\]
If we plug this result in the sum over $h_1$ and we repeat the
same estimate, we obtain that the sum over $h_1$ and $h_2$ is
bounded by $|k|^{-2}|k-m_1|^{-(2-\beta)}\log(1+|k|)$. 
By applying Lemma \ref{l:sumsum} to the sum over $m_1$,
we finally obtain that the term arising from the fifth class is
bounded by $|k|^{-(4-2\beta)}\log(1+|k|)$.
The third and fourth classes can be estimated by similar but simpler
considerations. In conclusion $\B(\zeta,\zeta)\in H^{-2-\gamma}$
for $\gamma>\beta-1$ and the lemma is proved.
\subsubsection{The term $\B(\zeta,z)$}

Given the information we have on the regularity of $z$ and
$\zeta$, we cannot use Corollary \ref{c:mixed} to give
a meaning to $\B(\zeta,z)$ (and it may not have a unique
meaning, as $\Bt(z,z)$). Hence we resort on the
same method we have used for $\Bt(z,z)$.
\begin{lemma}\label{l:addreg2}
  Let $\beta\in[\tfrac23,1)$ and let $z$, $\zeta$, $z^N$,
  $\zeta^N$ be defined as before. Then $(\B(\zeta^N,z^N))_{N\geq1}$
  is a Cauchy sequence in $L^2(\Omega;\Hz^{-2-\gamma})$ for every
  $\gamma>\tfrac\beta2$.
  
  Denote by $\Bt(\zeta,z)$ the limit of this sequence,
  which is well--defined as an element of $\Hz^{-2-\gamma}$,
  for $\gamma>\tfrac\beta2$. Given $\gamma>\tfrac\beta2$
  and $p\geq1$, there is a constant $c>0$ such that 
  \[
    \sup_{t>0}\E\bigl[\|\Bt(\zeta(t),z(t))\|_{H^{-2-\gamma}}^p\bigr]
      \leq c.
  \]
\end{lemma}
\begin{remark}\label{r:extend}
  When $\beta\in(0,\tfrac23)$, the term $\B(\zeta,z)$
  is well--defined and hence the above lemma ensures
  additional regularity for $\B(\zeta,z)$. The
  argument is the same of Lemma \ref{l:addreg1}.
\end{remark}

As in the previous part,
\[
\begin{multlined}[.9\linewidth]
  \B(\zeta^N,z^N)
    = \sum_{k,\ell\in\Z^2_\star}|k|^2\bar\phi_\ell^N\bar z_\ell
      \J_k^N\B(e_k,e_\ell) = {}\\
    = \sum_{k\in\Z^2_\star}|k|^2\Bigl(\sum_{h+\ell=k}
      |h|^2(h\cdot\ell) \phi_\ell^N z_\ell \J_h^N\Bigr)e_k
    = \sum_{k\in\Z^2_\star}|k|^2 G_k^N e_k,
\end{multlined}
\]
where $G_k^N$ is the inner sum, and compute the norm of
$\B(\zeta^N,z^N) - \B(\zeta^{N'},z^{N'})$ in
$H_\wp^{-2-\gamma}$ to obtain
\[
  \|\B(\zeta^N,z^N) - \B(\zeta^{N'},z^{N'})\|_{H^{-2-\gamma}}^2
    = \sum_{k\in\Z^2_\star}|k|^{-2\gamma}
      \bigl|G_k^N - G_k^{N'}\bigr|^2.
\]
Thus
\[
  \E[|G_k^N-G_k^{N'}|^2]
    \leq\sum_{\substack{h_1+\ell_1=k\\ h_2+\ell_2=k}}^{N\leftrightarrow N'}
       \frac{|\ell_1|\,|\ell_2|\,|\phi_{\ell_1}\bar\phi_{\ell_2}|}{|h_1||h_2|}
       \sup_{s_1,s_2}\bigl|\E\bigl[J_{h_1}(s_1)\bar J_{h_2}(s_2)
       z_{\ell_1}(t)\bar z_{\ell_2}(t)\bigr]\bigr|
\]
The term $\E[J_{h_1}\bar J_{h_2}z_{\ell_1}\bar z_{\ell_2}]$
contains sums of terms like $\E[z_{\ell_1}\bar z_{\ell_2}z_{m_1}
z_{n_1}\bar z_{m_2}\bar z_{n_2}]$ that, by
Wick's formula, are the sum of $15$ products of three pairwise
expectations. Five of these terms are $0$, by \eqref{e:simple},
since contain $\E[z_mz_n]$ with $m+n\neq0$. By symmetry
we can collect the terms in classes, whose representatives are
\[
\begin{gathered}
  \E[z_{\ell_1}\bar z_{\ell_2}]\E[z_{m_1}\bar z_{m_2}]\E[z_{n_1}\bar z_{n_2}]
  \qquad\E[z_{\ell_1}z_{m_1}]\E[\bar z_{\ell_2}\bar z_{m_2}]\E[z_{n_1}\bar z_{n_2}]\\
  \E[z_{\ell_1}\bar z_{m_2}]\E[\bar z_{\ell_2}z_{m_1}]\E[z_{n_1}\bar z_{n_2}].
\end{gathered}
\]
We focus on the first class. Fix $\eta>0$ small enough (depending
on $\beta$ and $\gamma$), then a simple computation,
the assumption on the $\phi_k$ and Lemma \ref{l:sumsum} yield
\[
\begin{multlined}[.9\linewidth]
  \sup\bigl|\E\bigl[J_{h_1}\bar J_{h_2}z_{\ell_1}\bar z_{\ell_2}\bigr]\bigr|
    \leq\frac{c\delta_{h_1-h_2}\delta_{\ell_1-\ell_2}}
      {|h_1|^{2}|\ell|^{2-\beta}}\Bigl(\sum_{m+n=h_1}
      \frac1{|m|^{2-\beta}|n|^{2-\beta}}\Bigr)\leq{}\\
    \leq \frac{c\delta_{h_1-h_2}\delta_{\ell_1-\ell_2}}
      {|h_1|^{4-2\beta}|\ell|^{2-\beta}}
    \leq \frac{c\delta_{h_1-h_2}\delta_{\ell_1-\ell_2}}
      {N^\eta|h_1|^{4-2\beta-\eta}|\ell|^{2-\beta-\eta}}.
\end{multlined}
\]
By summing in $h_1,h_2$, we obtain the estimate
$cN^{-\eta}|k|^{-(2-\beta-\eta)}$.
The other two terms can be estimated by similar
arguments giving the same bound. Hence the sequence
is Cauchy in $H_\wp^{-2-\gamma}$ for every
$\gamma>\tfrac\beta2$. The statement on
moments follows by hyper--contractivity as in
Proposition \ref{p:nelson}.
\subsubsection{The local mild solution}

The additional regularity of $\B(\zeta,\zeta)$ and the
interpretation of $\Bt(\zeta,z)$ we have proved, finally
allow to obtain a local solution of \eqref{e:hphi}
as $h = u+\zeta+z$, where $u$ is a mild solution
of \eqref{e:higher}.
\begin{proof}[Proof of Theorem \ref{t:mild3}]
  We give a quick sketch of the estimate for the ``self--mapping''
  property, the details are the same as Theorems \ref{t:mild}
  and \ref{t:mild2}. The terms arising from $\B(u,u)$ and
  $\B(u,z)$ in the mild formulation can be handled as
  $\B(v,v)$ and $\B(v,z)$, as well as $\B(u,\zeta)$,
  since $\zeta$ is smoother than $z$.
  For $\B(\zeta,\zeta)$ we use Lemma \ref{l:addreg1}
  and choose $q>4$ to get
  \[
    \Bigl\|\int_0^tS(t-s)\B(\zeta,\zeta)\,ds\Bigr\|_{H^{1+\epsilon}}
      \leq ct^{-\frac\epsilon4}T^{\frac14-\frac1q}
        \Bigl(\int_0^T\|\B(\zeta,\zeta)\|_{H^{-2}}\,ds\Bigr)^{\frac1q}.
  \]
  Likewise, by Lemma \ref{l:addreg2}, if $\epsilon\in(0,1-\tfrac\beta2)$
  and$\tfrac\beta2<\gamma<1-\epsilon$,
  \[
    \Bigl\|\int_0^tS(t-s)\Bt(\zeta,z)\,ds\Bigr\|_{H^{1+\epsilon}}
      \leq ct^{-\frac\epsilon4}T^{\frac14(1-\gamma)-\frac1q}
        \Bigl(\int_0^T\|\B(\zeta,z)\|_{H^{-2-\gamma}}\,ds\Bigr)^{\frac1q},
  \]
  for $q$ large enough.
\end{proof}
\appendix
\section{Bounds on the non--linear operator}

\begin{lemma}\label{l:Bsquare}
  If $\alpha,\beta,\gamma\geq0$ and $\alpha+\beta+\gamma\geq1$
  (with strict inequality if at least one of the numbers is equal to $1$),
  then $\B$ maps $H^{1+\alpha}_\wp\times H^{1+\beta}_\wp$
  continuously into $H^{-2-\gamma}_\wp$.  
  In particular, there exists $c = c(\alpha,\beta,\gamma)$ such that
  \[ 
    \|\B(u_1,u_2)\|_{H^{-2-\gamma}}
      \leq c\|u_1\|_{H^{1+\alpha}} \|u_2\|_{H^{1+\beta}}.
  \]
\end{lemma}
\begin{proof}
  Let $\phi\in H^{2+\gamma}_\wp$, then by integration by parts and
  the H\"older inequality,
  \[
    \scal{\phi, \B(u_1,u_2)}
      = \int \Delta\phi\nabla u_1\cdot\nabla u_2\,dx
      \leq \|\Delta\phi\|_{L^p}\|\nabla u_1\|_{L^q}\|\nabla u_2\|_{L^r},
  \]
  with $\tfrac1p+\tfrac1q+\tfrac1r=1$. Sobolev's embeddings yield
  $H^{1-2/p}_\wp\subset L^p$, hence
  \[
    \scal{\phi, \B(u_1,u_2)}
      \leq c\|\phi\|_{H^{3-2/p}}\|u_1\|_{H^{2-2/q}}\|u_2\|_{H^{2-2/r}}.
  \]
  Choose now $p,q,r$ such that $\alpha\geq1-\tfrac2q$,
  $\beta\geq1-\tfrac2r$ and $\gamma\geq1-\tfrac2p$. The case
  where one number is $1$ corresponds to the critical Sobolev
  embedding and needs a strict inequality. 
\end{proof}
\begin{proposition}
  Let $s\in(0,1)$, $s_1,s_2\in(s,1)$ and $p,p_1,p_2\geq1$ such that
  $\tfrac1{p_1}+\frac1{p_2}=\frac1p$. Then there is a constant $c>0$
  such that
  \[
    \|u_1u_2\|_{W^{s,p}}
      \leq c\|u_1\|_{W^{s_1,p_1}}\|u_2\|_{W^{s_2,p_2}},
  \]
  for all $u_1\in W^{s_1,p_1}_\wp$ and $u_2\in W^{s_2,p_2}_\wp$.
\end{proposition}
\begin{proof}
  By definition
  \[
    \|u_1u_2\|_{W^{s,p}}^p
      = \int |u_1u_2|^p\,dx
        + \iint\frac{|u_1(x)u_2(x)-u_1(y)u_2(y)|^p}{|x-y|^{2+sp}}\,dx\,dy.
  \]
  By H\"older's inequality,
  \[
    \int |u_1u_2|^p\,dx
      \leq \|u_1\|_{L^{p_1}}^p\|u_2\|_{L^{p_2}}^p.
  \]
  The second term in the norm above is split in the two terms
  \[
    \iint|u_1(x)|^p\frac{|u_2(x)-u_2(y)|^p}{|x-y|^{2+sp}}\,dx\,dy
        + \iint|u_2(y)|^p\frac{|u_1(x)-u_1(y)|^p}{|x-y|^{2+sp}}\,dx\,dy
      = \memo{a} + \memo{b}.
  \]
  Using again the H\"older inequality,
  \[
    \begin{multlined}[.95\linewidth]
      \memo{a}
        = \iint \frac{|u_1(x)|^p}{|x-y|^{\tfrac{2p}{p_1}-(s_2-s)p}}
            \Bigl(\frac{|u_2(x)-u_2(y)|^{p_2}}{|x-y|^{2+s_2p_2}}\Bigr)^{\frac{p}{p_2}}
            \,dx\,dy\leq {}\\
        \leq \Bigl(\frac{|u_2(x)-u_2(y)|^p_2}{|x-y|^{2+s_2p_2}}\Bigr)^{\frac{p}{p_2}}
             \Bigl(\frac{|u_1(x)|^p_1}{|x-y|^{2-(s_2-s)p_1}}\Bigr)^{\frac{p}{p_1}}
        \leq c\|u_1\|_{L^{p_1}}\|u_2\|_{W^{s_2,p_2}}
    \end{multlined}
  \]
  since $|x-y|^{(s_2-s)p-2}$ is integrable. The term \memo{b} can be
  estimated similarly.
\end{proof}
\begin{corollary}\label{c:mixed}
  Let $\epsilon\in(0,1)$ and $\gamma>0$. For every
  $\alpha\in(0,1)$ and $q>2$ such that
  $1-\alpha+\tfrac2q<\epsilon\wedge\gamma$,
  there is a constant $c>0$ such that
  \[
    \|\B(u_1,u_2)\|_{H^{-2-\gamma}}
      \leq c\|u_1\|_{W^{\alpha,q}}\|u_2\|_{H^{1+\epsilon}},
  \]
  for every $u_1\in W^{\alpha,q}_\wp$ and $u_2\in H^{1+\epsilon}_\wp$.
\end{corollary}
\begin{proof}
  Assume $\gamma\leq\epsilon$. By duality,
  \[
    \|\B(u_1,u_2)\|_{H^{-2-\gamma}}
      \leq c\|\nabla u_1\nabla u_2\|_{H^{-\gamma}}
      = c\sup_{\|\phi\|_{H^\gamma}}\scal{\phi,\nabla u_1\nabla u_2}.
  \]
  Let $p$ be the conjugate exponent of $q$, then by integration by parts,
  duality and the previous proposition (with $p_1=2$, $s_1=\gamma$
  and $1-\alpha<s_2\leq\epsilon-\tfrac2q$),
  \[
    \begin{multlined}[.9\linewidth]
      \scal{\phi,\nabla u_1\nabla u_2}
        = -\scal{u_1, \Div(\phi\nabla u_2)}
        \leq c\|u_1\|_{W^{\alpha,q}}\|\phi\nabla u_2\|_{W^{1-\alpha,p}}\leq{}\\
        \leq c\|u_1\|_{W^{\alpha,q}}\|\phi\|_{H^{\gamma}}\|\nabla u_2\|_{W^{s_2,p_2}}
        \leq c\|u_1\|_{W^{\alpha,q}}\|u_2\|_{H^{1+\epsilon}},
    \end{multlined}
  \]
  since $H^\epsilon_\wp\subset W^{s_2,p_2}_\wp$ by the choice
  of $s_2$ and $p_2$.
  
  If on the other hand $\epsilon\leq\gamma$, apply the previous
  proposition with $p_2=2$, $s_2=\epsilon$, and
  $1-\alpha<s_2\leq\gamma-\tfrac2q$, and use the Sobolev embedding
  $H^\gamma_\wp\subset W^{s_1,p_1}_\wp$.
\end{proof}
The following Lemma is stated in \cite[Lemma 2.3]{BloRomTri2007}. 
\begin{lemma}\label{l:sumsum}
  For all $\alpha$, $\gamma >0$ with $\alpha+\gamma>d$ there exists
  $C = C(\alpha, \gamma) < \infty$ so that, for all $k \in \Z^d$,
  with $k\neq\mathbf{0}$,
  \[
    \sum_{\substack{m+n=k\\m\neq0,n\neq0}}\frac1{|m|^\alpha|n|^\gamma}\leq
      \begin{cases}
        C(1+|k|)^{-\beta},
          &\qquad\text{if $\alpha\neq d$ and $\gamma\neq d$},\\
        C(1+|k|)^{-\beta}\log(1+|k|),
          &\qquad\text{if $\alpha=d$ or $\gamma=d$},
      \end{cases}
  \]
  where $\beta=\min\{\alpha, \gamma, \alpha+\gamma-d\}$.
\end{lemma}
\def\MRnum#1
  #2\empty{#1}\renewcommand{\MRhref}[2]{\href{http://www.ams.org/mathscinet-getitem?mr=#1}{#2}}\renewcommand{\MR}[1]{\relax\ifhmode\unskip\space\fi
  \MRhref{\MRnum#1\empty}{\texttt{\Tiny[MR\MRnum#1\empty]}}}
  \providecommand{\arxiv}[1]{\href{http://arxiv.org/abs/#1}{arXiv:#1}}
\providecommand{\bysame}{\leavevmode\hbox to3em{\hrulefill}\thinspace}
\providecommand{\MR}{\relax\ifhmode\unskip\space\fi MR }
\providecommand{\MRhref}[2]{%
  \href{http://www.ams.org/mathscinet-getitem?mr=#1}{#2}
}
\providecommand{\href}[2]{#2}
\bibliographystyle{amsplain}

\end{document}